\documentclass[11pt]{article}
\usepackage[T1]{fontenc}
\usepackage{graphicx}
\usepackage{amsmath}
\usepackage{amsthm}
\usepackage{amssymb}
\usepackage{stmaryrd}
\usepackage{todonotes}
\usepackage[curve]{xypic}
\usepackage[left=3.2cm,right=3.2cm,top=3cm,bottom=3cm]{geometry} 
\usepackage{tikz}
\usetikzlibrary{calc, decorations.markings, fit, external, chains, graphs, intersections,through,backgrounds, calc}
\usepackage{tikz-cd}
\usepackage{wrapfig}
\usepackage{enumerate}
\usepackage[pagebackref]{hyperref}

\definecolor{dark-red}{rgb}{0.5,0.15,0.15}
\definecolor{dark-blue}{rgb}{0.15,0.15,0.6}
\definecolor{dark-green}{rgb}{0.15,0.6,0.15}
\hypersetup{
	colorlinks, linkcolor=dark-red,
	citecolor=dark-blue, urlcolor=dark-green
}
\usepackage[nameinlink,capitalise,noabbrev]{cleveref}

\numberwithin{equation}{section}
\newtheorem{thm}[equation]{Theorem}

\newtheorem*{thm*}{Theorem}
\newtheorem*{mthm*}{Main Theorem}

\newtheorem{lemma}[equation]{Lemma}

\newtheorem{cor}[equation]{Corollary}

\newtheorem{prop}[equation]{Proposition}

\newtheorem*{conjecture*}{Conjecture}

\newtheorem*{question*}{Question}

\theoremstyle{definition}
\newtheorem{defi}[equation]{Definition}

\newtheorem{notation}[equation]{Notation}

  \newtheorem*{example*}{Example}

  \newtheorem*{exe*}{Exercise}

\theoremstyle{remark}

\newtheorem{remark}[equation]{Remark}

\newtheorem{rmk}[equation]{Remark}
\newtheorem{construction}[equation]{Construction}

\newcommand{\G}{\mathbb{G}}
\newcommand{\Z}{\mathbb{Z}}
\newcommand{\F}{\mathbb{F}}

\newcommand{\Q}{\mathbb{Q}}

\newcommand{\R}{\mathbb{R}}
\newcommand{\C}{\mathbb{C}}

\renewcommand{\S}{\mathbb{S}}

\newcommand{\CP}{\mathbb{CP}}

\newcommand{\CC}{\mathcal{C}}

\newcommand{\EE}{\mathcal{E}}

\newcommand{\LL}{\mathcal{L}}

\newcommand{\XX}{\mathcal{X}}

\newcommand{\OO}{\mathcal{O}}

\newcommand{\MM}{\mathcal{M}}

\newcommand{\MMb}{\overline{\mathcal{M}}}

\newcommand{\tensor}{\otimes}

\DeclareMathOperator{\pt}{pt}

\DeclareMathOperator{\Sp}{Sp}

\DeclareMathOperator{\Spf}{Spf}

\DeclareMathOperator{\Ab}{Ab}

\DeclareMathOperator{\id}{id}

\DeclareMathOperator{\pr}{pr}

\DeclareMathOperator{\Tor}{Tor}

\DeclareMathOperator{\Ext}{Ext}
\DeclareMathOperator{\Hom}{Hom}
\DeclareMathOperator{\Spec}{Spec}
\DeclareMathOperator{\fib}{fib}

\DeclareMathOperator{\QCoh}{QCoh}

\DeclareMathOperator{\Mod}{Mod}

\DeclareMathOperator{\res}{res}

\DeclareMathOperator{\Sch}{Sch}

\DeclareMathOperator{\Map}{Map}
\DeclareMathOperator{\CAlg}{CAlg}

\renewcommand{\P}{\mathbb{P}}
\renewcommand{\top}{\mathrm{top}}

\DeclareMathOperator{\SL}{SL}
\DeclareMathOperator{\tmf}{tmf}
\DeclareMathOperator{\Tmf}{Tmf}
\DeclareMathOperator{\TMF}{TMF}

\begin{document}
\title{Connective Models for Topological Modular Forms of Level $n$}
\author{Lennart Meier}
\maketitle

\begin{abstract}
	The goal of this article is to construct and study connective versions of topological modular forms of higher level like $\tmf_1(n)$. In particular, we use them to realize Hirzebruch's level-$n$ genus as a map of ring spectra. 
\end{abstract}

\tableofcontents

\section{Introduction}
The basic tenet of Waldhausen's philosophy of \emph{brave new algebra} is to replace known notions for commutative rings by corresponding notions for $E_\infty$-ring spectra. These days replacing the integers by the sphere spectrum is actually no longer so brave and new, but rather a well-established principle. In extension, we might want to find and study $E_\infty$-analogues of other prominent rings as well. The aim of the present paper is to do this for rings of holomorphic modular forms with respect to congruence subgroups of $\SL_2(\Z)$.

Topological analogues of modular forms for $\SL_2(\Z)$ itself were already introduced about twenty years ago. Indeed, Goerss, Hopkins and Miller introduced three spectra $\TMF$, $\Tmf$ and $\tmf$ of topological modular forms. Recall that the rings $M_*(\SL_2(\Z);\Z)$ and $\widetilde{M}_*(\SL_2(\Z);\Z)$ of holomorphic and meromorphic integral modular forms can be defined as the global sections $H^0(\MMb_{ell}; \omega^{\tensor *})$ and $H^0(\MM_{ell}; \omega^{\tensor *})$ of powers of a certain line bundle $\omega$ on the compactified and uncompactified moduli stack of elliptic curves, respectively.\footnote{The terms \emph{meromorphic} and \emph{holomorphic} come from the corresponding analytic definitions, where one demands that the given function on the upper half plane can be continued meromorphically and holomorphically, respectively, to the cusp(s). The former kind of modular forms is also sometimes called \emph{weakly holomorphic}.} In analogy, $\TMF$ is defined as the global sections of a sheaf $\OO^{top}$ of $E_{\infty}$-ring spectra on $\MM_{ell}$ with $\pi_{2k}\OO^{top} \cong \omega^{\tensor k}$ and $\Tmf$ as the global sections of an analogous sheaf on $\MMb_{ell}$. The edge maps of the resulting descent spectral sequences take the form of homomorphisms
\begin{align*}
\pi_{2*}\TMF &\to \widetilde{M}_*(\SL_2(\Z); \Z) \text{ and}\\
\pi_{2*}\Tmf &\to M_*(\SL_2(\Z); \Z).
\end{align*}
The former morphism is an isomorphism after base change to $\Z[\frac16]$ (while taking higher cohomology of $\omega^{\tensor *}$ into account at the primes $2$ and $3$) and thus $\TMF$ can be really seen as the rightful analogue of $\widetilde{M}(\SL_2(\Z); \Z)$. In contrast, $\pi_*\Tmf$ has torsionfree summands in negative degree, whereas $M_*(\SL_2(\Z), \Z)$ is concentrated in non-negative degrees. The solution is to define $\tmf$ simply as the connective cover $\tau_{\geq 0}\Tmf$ and one can show that indeed $\pi_{2*}\tmf[\frac16]$ is isomorphic to $M_*(\SL_2(\Z), \Z[\frac16])$. We mention that one of the motivations for constructing $\tmf$ was lifting the Witten genus to a map of $E_{\infty}$-ring spectra $MString \to \tmf$ as achieved in \cite{AHR10}. For applications to the stable homotopy groups of spheres and exotic spheres see e.g.\ \cite{H-M98}, \cite{BHHM}, \cite{W-X17} and \cite{MoreStableStems}. 

In number theory, it is very common not only to consider modular forms with respect to $\SL_2(\Z)$, but also to congruence subgroups of these; the most important being $\Gamma = \Gamma_0(n)$, $\Gamma_1(n)$ or $\Gamma(n)$. Algebro-geometrically, such modular forms can be defined as sections of the pullback of $\omega^{\tensor *}$ to compactifications $\MMb(\Gamma)$ of stacks classifying generalized elliptic curves with certain level structures (see e.g.\ \cite{D-R73}, \cite{D-I95}, \cite{Con07}, \cite{MeiDecMod}); for example, $\MMb(\Gamma_1(n))$ classifies generalized elliptic curves with a chosen point of order $n$ whose multiples intersect every irreducible component of every geometric fiber. Hill and Lawson \cite{HL13} defined sheaves of $E_{\infty}$-ring spectra on these stacks and obtained spectra $\Tmf(\Gamma)$, as their global sections, and $\TMF(\Gamma)$, by restriction to the loci of smooth elliptic curves. The latter spectra are good topological analogues of the rings $\widetilde{M}(\Gamma; \Z[\frac1n])$ of meromorphic modular forms in the sense that $\pi_*\TMF(\Gamma)$ is isomorphic to this ring if $\Gamma$ is $\Gamma_1(n)$ or $\Gamma(n)$ (with $n\geq 2$) and, if we invert $6$, also in the case $\Gamma = \Gamma_0(n)$.

In contrast, neither $\Tmf(\Gamma)$ nor its connective cover $\tau_{\geq 0}\Tmf(\Gamma)$ are in general good analogues of the ring of holomorphic modular forms $M(\Gamma; \Z[\frac1n])$, even in the nice case of $\Gamma = \Gamma_1(n)$ and $n\geq 2$. Writing $\Tmf_1(n)$ for $\Tmf(\Gamma_1(n))$, the reason is that $H^1(\MMb(\Gamma_1(n)); \omega)$ and thus $\pi_1\Tmf_1(n)$ is non-trivial in general (with $n=23$ being the first example), while this contribution does not occur in $M(\Gamma; \Z[\frac1n])$. Following an idea of Lawson, we define a connective version $\tmf_1(n)$ by ``artificially'' removing $\pi_1$, while still retaining the $E_{\infty}$-structure on $\tmf_1(n)$. The following will be proven as \cref{thm:tmfGamma} and \cref{prop:tmf1(n)C2}.
\begin{thm}\label{MainTheorem1}
		There is an essentially unique connective $E_{\infty}$-ring spectrum $\tmf_1(n)$ with an $E_{\infty}$-ring map $\tmf_1(n) \to \Tmf_1(n)$ that identifies the homotopy groups of the source with $M(\Gamma_1(n);\Z[\frac1n])$. 
		
		Moreover, the involution of $\MMb(\Gamma_1(n))$ sending a point of order $n$ on the universal elliptic curve to its negative defines on $\tmf_1(n)$ the structure of a genuine $C_2$-spectrum. Its slices in the sense of \cite{HHR} are trivial in odd degrees and can be explicitly identified in even degrees. 
\end{thm}
The analogous theorem also works to define $\tmf(n)$, but $\tmf_0(n)$ we define only in certain cases since in the general case it is not yet clear what the ``correct'' definition is. The spectrum $\tmf(n)$ has been further investigated in \cite[Theorem 3.14]{H-RRamification}, where a criterion for the non-vanishing of its Tate spectrum is proven. 

One of the principal motivations for the consideration of $\tmf_1(n)$ is its connection to the Hirzebruch level-$n$ genera $MU_* \to M(\Gamma_1(n); \Z[\frac1n])$. They specialize for $n=2$ to the classic Ochanine elliptic genus and have similar rigidity properties in general \cite{HBJ}. We will prove the following as \cref{thm:HirzebruchRealization}.

\begin{thm}
	For every $n\geq 2$, there is a ring map $MU \to \tmf_1(n)$ realizing on homotopy groups the Hirzebruch level-$n$-genus. Moreover, this map refines to a map $MU_{\R} \to \tmf_1(n)$ of $C_2$-spectra. 
\end{thm}

We have two further classes of results on the spectra $\tmf_1(n)$ and their cousins. The first is the following compactness result, contained in \cref{thm:tmfperfect} and \cref{cor:fp}.
\begin{thm}
	The $\tmf[\frac1n]$-modules $\tmf_0(n)$, $\tmf_1(n)$ and $\tmf(n)$ are perfect, i.e.\ they are compact objects in the module category, in the cases they are defined. In particular, their $\F_p$-cohomologies are finitely presented over the Steenrod algebra and thus their $p$-completions are fp-spectra in the sense of \cite{M-R99}.
\end{thm}
By a result of Kuhn \cite[Theorem 1.7]{KuhnHurewicz} this implies for example that the Hurewicz image of $\pi_*\tmf(\Gamma) \cong \pi_*\Omega^{\infty}\tmf(\Gamma)$ in $H_*(\Omega^{\infty}\tmf(\Gamma); \F_p)$ is finite dimensional, where $\tmf(\Gamma)$ denotes either $\tmf_0(n)$, $\tmf_1(n)$ or $\tmf(n)$. We also note that in contrast to the theorem, $\tmf_1(n)$ will not be a perfect $\tmf_0(n)$-module in general. We also show that $\tmf_0(n)$, $\tmf_1(n)$ and $\tmf(n)$ are faithful as $\tmf[\frac1n]$-modules, answering a question of H{\"o}ning and Richter \cite[p.21]{H-RRamification}.

The second result is a variant of the decomposition results of \cite{MeiTopLevel}, which we state in this introduction only at the prime $2$ and for $\tmf_1(n)$, and which will be proven as \cref{thm:splittingC2}.
\begin{thm}
Let $n>1$ be odd. If one can lift every weight $1$-modular form for $\Gamma_1(n)$ over $\F_2$ to a form of the same weight and level over $\Z_{(2)}$, we have a $C_2$-equivariant splitting
\[\tmf_1(n)_{(2)} \simeq \bigoplus_i \Sigma^{n_i\rho}\tmf_1(3)_{(2)},\]
where $\rho$ denotes the real regular representation of $C_2$. 
\end{thm} 
In \cite[Appendix C]{MeiDecMod}, it is shown that for $1<n<65$ odd indeed every weight-$1$-modular form for $\Gamma_1(n)$ over $\F_2$ lifts to a form of the same weight and level over $\Z_{(2)}$, while for $n=65$ it does not. See also \cite[Remark 3.14]{MeiDecMod} for a further discussion of this condition.

\subsection*{Acknowledgments}
I want to thank Tyler Lawson for explaining to me the idea how to construct a connective model for $\TMF_1(n)$ and for the sketch of an argument that $\tmf_1(3)$ is not perfect as a $\tmf_0(3)$-module. It is also a pleasure to acknowledge the influence and encouragement of Mike Hill. Furthermore I want to thank Eva H{\"o}ning and Birgit Richter for their interest and remarks on a preliminary version and the anonymous referee for their extensive comments. 

Finally, I want to thank the Hausdorff Institute for hospitality in 2015 when part of this work was undertaken. Apologies for the subsequent delay in publication.

\subsection*{Conventions and notation}
All notions are to be understood suitably derived or $\infty$-categorical. This means that pushout means either a pushout in the respective $\infty$-category or a homotopy pushout in the underlying model category. We will use $\otimes$ for the (derived) smash product. Note that this coincides with the coproduct in the $\infty$-category $\CAlg$ of $E_{\infty}$-ring spectra. 

When we use $G$-spectra, we will always mean genuine $G$-spectra. The notations $\tau_{\leq k}$ and $\tau_{\geq k}$ denote the $k$-(co)connective cover of a spectrum  and we use the same notation for the slice-(co)connective covers of a $G$-spectrum. Furthermore, we denote by $\S$ the sphere ($G$-)spectrum. In some parts of this article, we have the opportunity to use $RO(C_2)$-graded homotopy groups of $C_2$-spectra. We will use the notation $\sigma$ for the sign representation and $\rho$ or $\C$ for the regular representation of $C_2$. 

We will use the notations $\TMF_1(n)$ and $\TMF(\Gamma_1(n))$ interchangeably and similarly in related contexts. 

\section{The construction of connective topological modular forms}\label{sec:constr}
The aim of this section is to construct connective spectra $\tmf(\Gamma)$ of topological modular forms and thereby prove \cref{MainTheorem1}. Here $\Gamma$ denotes a congruence subgroup $\Gamma$ in the following sense, which is a bit more restrictive than the standard definition. 
\begin{defi}
	We call $\Gamma\subset SL_2(\Z)$ a \emph{congruence subgroup of level $n$} if $\Gamma = \Gamma(n)$ or $\Gamma_1(n) \subset \Gamma \subset \Gamma_0(n)$.\footnote{We refer to \cite{D-S05} for background on the congruence subgroups $\Gamma_1(n)$, $\Gamma(n)$ and $\Gamma_0(n)$ and their relationship to moduli of elliptic curves. This material is though barely necessary for the present paper as we use the congruence subgroups primarily as notation.}
\end{defi}
As explained in \cite{HL13} and \cite[Section 2.1]{MeiTopLevel}, we can associate with every such $\Gamma$ a (non-connective and non-periodic) $E_{\infty}$-ring spectrum $\Tmf(\Gamma)$. (See also \cite[Theorem 5.2]{Sto12} for the case of $\Gamma(n)$.) These arise as global sections of sheaves of $E_{\infty}$-ring spectra $\OO^{top}$ on stacks $\MMb(\Gamma)$ classifying generalized elliptic curves with certain level structures; the details will not be important for the purposes of this article, but see e.g.\ \cite{D-R73}, \cite{Con07}, \cite{Ces17}, \cite{MeiDecMod}. Our goal in this section is to construct a nice connective version $\tmf(\Gamma)$ for $\Tmf(\Gamma)$. For this, we will fix a localization $\Z_S$ of the integers and restrict mostly to tame congruence subgroups.
\begin{defi}We say that a congruence subgroup $\Gamma$ of level $n$ is \emph{tame} with respect to $\Z_S$ if $n\geq 2$ and $n$ is invertible in $\Z_S$; in the case $\Gamma_1(n)\subset \Gamma\subset \Gamma_0(n)$ we demand additionally that $\gcd(6,[\Gamma:\Gamma_1(n)])$ is invertible in $\Z_S$.\footnote{As the quotient $\Gamma_0(n)/\Gamma_1(n)$ is $(\Z/n)^\times$, the latter condition reduces to $\gcd(6, \varphi(n))$ being invertible in the case $\Gamma = \Gamma_0(n)$. Thus we require that $2$ is invertible and also $3$ if $n$ is divisible by a prime of the form $3k+1$ or by $9$.} \end{defi}
The definition ensures that the order of every automorphism of a point in $\MMb(\Gamma)$ is invertible and thus the stack is of cohomological dimension $1$. As explained in \cite[Section 2.1]{MeiTopLevel}, in this case $\pi_* \tau_{\geq 0}\Tmf(\Gamma)$ is concentrated in even degrees except for $\pi_1 \Tmf(\Gamma)$, which might be nonzero. (The smallest $n$ for which this happens is $23$.) Moreover, the even homotopy groups of $\Tmf(\Gamma)$ are precisely isomorphic to the ring of holomorphic modular forms $M(\Gamma; \Z[\frac1n])$. 

Following the lead of \cite[Proposition 11.1]{Law15} (and additional explanations by its author), we will first describe a general procedure to kill $\pi_1$ for $E_{\infty}$-rings that applies to $\tau_{\geq 0}\Tmf(\Gamma)$ for $\Gamma$ tame. We will then present a $C_2$-equivariant refinement that helps to define a nice version of $\tmf(\Gamma)$ also in some non-tame case (see \cref{constr:Gamma'}). We note that our techniques are only necessary if $\pi_1\Tmf(\Gamma)$ is non-trivial as else the usual connective cover defines a perfectly good version of $\tmf(\Gamma)$. 

\subsection{The non-equivariant argument}
Let $R$ be a connective $E_\infty$-ring spectrum with $\pi_0R$ an \'etale extension of $\Z_{S}$, a localization of $\Z$, and $\eta\cdot 1 = 0$; here, $\eta \in \pi_1\mathbb{S}$ is the Hopf element and $1\in \pi_0R$ the unit. (The relevant example for us is $R = \tau_{\geq 0}\Tmf(\Gamma)_S$ with $\pi_0R = \Z_S$ if $\Gamma_1(n)\subset \Gamma \subset \Gamma_0(n)$ and $\pi_0R = \Z_S[\zeta_n]$ if $\Gamma = \Gamma(n)$.) We want to construct a map $R' \to R$ of $E_\infty$-ring spectra, which is injective on $\pi_*$ and with cokernel $\pi_1R$. In the following, we localize everything implicitly at the set $S$. In particular, $\Z$ really means $\Z_S$ etc. 

Let $A$ first be a general $E_\infty$-ring spectrum. For an $A$-module $M$, we denote by 
\[\P_A(M) \simeq A \oplus M \oplus (M^{\tensor_A 2})_{h\Sigma_2} \oplus \cdots\]
the free unital $E_{\infty}$-$A$-algebra on $M$ (cf.\ \cite[3.1.3.14]{HA}). 
\begin{defi}
	Let $x\colon \Sigma^kA \to A$ be an $A$-linear map. We define its \emph{$E_\infty$-cone} $C^A(x)$ as the pushout $A\tensor_{\P_A(\Sigma^kA)} A$ of $E_{\infty}$-ring spectra. Here, the first map $\P_A(\Sigma^kA) \to A$ is the free $E_{\infty}$-map on $x$, while the second arises from applying $\P_A$ to the unique map $\Sigma^kA \to 0$. 
\end{defi}
Note that if $B$ is an $E_{\infty}$-$A$-algebra, we have $C^A(x) \tensor_A B \simeq C^B(x)$. Writing the usual cone $C(x)$ as the pushout $A \sqcup_{\Sigma^kA \oplus A} A$ in $A$-modules
  produces a map $C(x) \to C^A(x)$ via the inclusion $A \oplus \Sigma^kA \to \P^A(\Sigma^kA)$ of the first two summands and the identity $\id_A$. 
  \begin{lemma}\label{lem:splitinjective}
  	If $x=0$, the canonical map $C(x) \to C^A(x)$ is split as a map of $A$-modules. 
  \end{lemma}
\begin{proof}
	The pushout square
	\begin{align} \label{eq:sqzero}
	\xymatrix{A \oplus \Sigma^kA \ar[r]\ar[d] & A \ar[d] \\
		A \ar[r] & C(0) \simeq A \oplus \Sigma^{k+1}A
	}
	\end{align}
	arises from the pushout square
	\begin{align}\label{eq:trivial}
	\xymatrix{\Sigma^kA \ar[r]\ar[d] & 0 \ar[d] \\
	0 \ar[r] & \Sigma^{k+1}A}
	\end{align}
	via the functor $\Mod_A \to \CAlg_A$ of square-zero extension. In particular, it is a diagram of $E_{\infty}$-$A$-algebras. As the $E_{\infty}$-pushout square (P) defining $C^A(0)$ arises from \eqref{eq:trivial} as well, but via $\P_A$, we see that the square \eqref{eq:sqzero} receives a map from the square (P). The resulting map $C^A(0) \to C(0)$ defines a splitting of $C(0) \to C^{A}(0)$ by the universal property of the pushout square \eqref{eq:sqzero}. 
\end{proof}
   
We will apply our general consideration to the connective $E_{\infty}$-ring spectrum $R$ we have fixed. As $\eta$ is zero in $\pi_*R$, we obtain an $E_\infty$-map $C^{\S}(\eta) \to R$. This induces an $E_\infty$-map $\tau_{\leq 1} C^{\S}(\eta) \to \tau_{\leq 1}R$ (see \cite[Proposition 4.35]{HHR}). 

\begin{lemma}\label{lem:connectivity}
The $1$-coconnective cover $\tau_{\leq 1}C^{\S}(\eta)$ is equivalent to $H\Z$. 
\end{lemma}
\begin{proof}
We claim that the canonical map $C(\eta) \to C^{\S}(\eta)$ is $2$-connected. By the Hurewicz theorem, we can test this after tensoring with $H\Z$ and thus it suffices
 to show that the resulting map $C(\eta \tensor H\Z) \to C^{H\Z}(\eta\tensor H\Z)$ is $2$-connected. 
 But $\eta \tensor H\Z$ agrees with the $0$-map $\Sigma H\Z \to H\Z$. Thus, we have to show that
\[H\Z \oplus \Sigma^2H\Z \to C^{H\Z}(\eta \tensor H\Z) \simeq \P^{H\Z}{\Sigma^2H\Z} \simeq H\Z \oplus \Sigma^2H\Z \oplus (\Sigma^4H\Z)_{hC_2}\oplus \cdots \]
is $2$-connected. As noted above, the map is split injective and thus must be indeed an isomorphism on $\pi_i$ even for $i\leq 3$. 
\end{proof}

%
%
%
The inclusion of $1$-truncated connective $E_{\infty}$-ring spectra into all connective $E_{\infty}$-ring spectra admits a left adjoint by \cite[Proposition 5.5.6.18]{HTT} and \cite[Proposition 7.1.3.14]{HA}. By \cite[Proposition 7.1.3.15(3)]{HA} it agrees with $\tau_{\leq 1}$ on underlying spectra. 

By \cite[Theorem 7.5.0.6]{HA}, we can extend the $E_{\infty}$-map $H\Z = \tau_{\leq 1}C^{\mathbb{S}}(\eta) \to \tau_{\leq 1}R$ to an $E_{\infty}$-map $H\pi_0R \to \tau_{\leq 1}R$ since the map $\Z \to \pi_0R$ is \'etale. Define now $R'$ via the homotopy pullback square
\begin{equation}\label{eq:square}\xymatrix{R'\ar[r]\ar[d] & H\pi_0R \ar[d] \\
R \ar[r] & \tau_{\leq 1}R} 
\end{equation}
This construction provides the existence part of the following proposition. 

\begin{prop}\label{prop:killingpi1}
	Let $R$ be a connective $E_{\infty}$-ring spectrum such that $\pi_0R$ is an \'etale extension of a localization $\Z_S$ of the integers and $\eta\cdot 1 = 0$ in $\pi_1R$. Then there exists a morphism $R' \to R$ of $E_{\infty}$-ring spectra inducing an isomorphism on $\pi_i$ for $i\neq 1$ and satisfying $\pi_1R' = 0$. Moreover, for every other $R'' \to R$ with these properties, there is an equivalence $R''\to R'$ of $E_{\infty}$-ring spectra over $R$. 
\end{prop}
\begin{proof}
	It remains to show uniqueness. We localize again everything implicitly at $S$. We first note that the map $H\Z \to \tau_{\leq 1} R$ constructed above is actually the unique $E_{\infty}$-map with this source and target. Indeed: For connectivity reasons, we have an equivalence of mapping spaces $\Map_{\CAlg}(H\Z, \tau_{\leq 1} R) \simeq \Map_{\CAlg}(C^{\S}(\eta), \tau_{\leq 1} R)$. The latter is equivalent to the space of nullhomotopies of $\eta$ in $\tau_{\leq 1}R$, i.e.\ $\Map_{\Sp}(\Sigma^2\mathbb{S}, \tau_{\leq 1}R) \simeq \ast$. Using that thus $\tau_{\leq 1} R$ has an essentially unique structure of an $H\Z$-$E_{\infty}$-algebra, we deduce again from \cite[Theorem 7.5.0.6]{HA} that the space of $E_{\infty}$-maps from $H\pi_0R$ to $\tau_{\leq 1}R$ is equivalent to the set of ring homomorphisms $\pi_0R \to \pi_0R$.
	
	Given now $R'' \to R$ as in the proposition, we obtain a map $R'' \to \tau_{\leq 1}R'' \simeq H\pi_0R \to \tau_{\leq 1}R$. We see that $R''$ arises as a pullback of a diagram of the same shape as \eqref{eq:square}, but possibly with a map $H\pi_0R \to \tau_{\leq 1}R$ inducing a different isomorphism $f$ on $\pi_0$ than the identity. The paragraph above implies that using the map $f$ on $H\pi_0R$ we obtain an equivalence between the cospans constructing $R'$ and $R''$ and thus between $R'$ and $R''$ over $R$.
\end{proof}

To apply this to topological modular forms, we need the following two lemmas.
\begin{lemma}\label{lem:eta}
	Let $\Gamma$ be a tame congruence subgroup with respect to a localization $\Z_S$. Then $\eta$ is zero in $\pi_1\Tmf(\Gamma)_S$.
\end{lemma}
\begin{proof}
According to \cite[Proposition 2.5]{MeiTopLevel}, the descent spectral sequence 
\[H^s(\MMb(\Gamma); \omega^{\tensor t}) \Rightarrow \pi_{2t-s}\Tmf(\Gamma) \]
for $\Tmf(\Gamma)_S$ is concentrated in lines $0$ and $1$. Thus $\pi_1\Tmf(\Gamma)_S \cong H^1(\MMb(\Gamma)_S; \omega)$ and it suffices to show that the image of $\eta$ in $H^1(\MMb(\Gamma)_S; \omega)$ is trivial. This is the content of \cite[Proposition 2.16]{MeiDecMod} unless $\Gamma_1(n) \subsetneq \Gamma \subsetneq \Gamma_0(n)$. For the reminder of the proof, assume that we are in this case and set $G = \Gamma/\Gamma_1(n)$.

We will argue that the map 
\[H^1(\MMb(\Gamma)_{(2)}; \omega) \to H^1(\MMb(\Gamma_1(n)_{(2)}; \omega)\]
is isomorphic to the inclusion of $G$-invariants. As $\eta$ vanishes in the target, this will imply the vanishing of $\eta$ in the source. 

The map $\MMb(\Gamma_1(n))_{(2)} \to \MMb(\Gamma)_{(2)}$ induces a map $c\colon \XX = [\MMb(\Gamma_1(n))_{(2)}/G] \to \MMb(\Gamma)_{(2)}$ from the stack quotient. Denote the pullback of $\omega$ to $\XX$ also by $\omega$. By \cite[Lemma A.2]{MeiTopLevel}, the induced map $H^1(\MMb(\Gamma)_{(2)};\omega) \to H^1(\XX; \omega)$ is an isomorphism. Moreover, the descent spectral sequence
\[H^i(G; H^j(\MMb(\Gamma_1(n))_{(2)}; \omega)) \Rightarrow H^{i+j}(\XX; \omega) \]
is concentrated in the zero-line since the order of $G$ is invertible in $\Z_{(2)}$ by the tameness of $\Gamma$. Thus, 
\[H^1(\MMb(\Gamma)_{(2)}; \omega)\cong H^1(\XX; \omega) \to H^1(\MMb(\Gamma_1(n)_{(2)}; \omega)\]
is indeed the inclusion of $G$-invariants. 
\end{proof}

\begin{lemma}
Let $\Gamma$ be a tame congruence subgroup with respect to a localization $\Z_S$. Then $\pi_0\Tmf(\Gamma) \cong \Z_S$ if $\Gamma_1(n) \subset \Gamma \subset \Gamma_0(n)$ and $\pi_0\Tmf(\Gamma) \cong \Z_S[\zeta_n]$ if $\Gamma = \Gamma(n)$.
\end{lemma}
\begin{proof}
As recalled above, we have $\pi_0\Tmf(\Gamma) \cong H^0(\MMb(\Gamma); \OO_{\MMb(\Gamma)})$. In the cases that $\Gamma = \Gamma_0(n), \Gamma_1(n)$ or $\Gamma(n)$ the computation of this group is classical and can be found e.g.\ in \cite[Proposition 2.13]{MeiDecMod}. The case of $\Gamma_1(n) \subsetneq \Gamma \subsetneq \Gamma_0(n)$ follows by identifying $H^0(\MMb(\Gamma); \OO_{\MMb(\Gamma)})$ with $H^0(\MMb(\Gamma); \OO_{\MMb(\Gamma)})^{\Gamma/\Gamma_1(n)}$ using \cite[Lemma A.2]{MeiTopLevel} again.
\end{proof}

This allows us to use \cref{prop:killingpi1} to define $\tmf(\Gamma)_S$ in the tame case by killing $\pi_1$ from $\tau_{\geq 0}\Tmf(\Gamma)_S$. Summarizing we obtain:
\begin{thm}\label{thm:tmfGamma}
	For every set of primes $S$ and every congruence subgroup $\Gamma$ that is tame with respect to $\Z_S$, there is up to equivalence a unique connective $E_{\infty}$-ring spectrum $\tmf(\Gamma)_S$ with an $E_{\infty}$-ring map $\tmf(\Gamma)_S \to \Tmf(\Gamma)_S$ that identifies the homotopy groups of the source with the ring of holomorphic modular forms $M(\Gamma; \Z_S)$. 
\end{thm}
Formally, we could also apply this procedure in some non-tame cases (e.g.\ if we localize away from $2$), but the author knows of no reason to regard these constructions in these cases as ``correct''. 

\begin{notation}
	We will use the abbreviations
	\begin{eqnarray*}
		\tmf_1(n) &= \tmf(\Gamma_1(n)) \\
		\tmf_0(n) &= \tmf(\Gamma_0(n)) \\
		\tmf(n) &= \tmf(\Gamma(n))
	\end{eqnarray*}
when these make sense. 
\end{notation}

\begin{remark}
	For every ring spectrum $R$, we can consider the stack $\XX_R$ associated to the graded Hopf algebroid $(MU_{2*}(R), (MU\tensor MU)_{2*}(R))$. If $R$ is complex orientable, this coincides with the stack quotient $[\Spec \pi_{2*}R/\G_m]$. In \cite[Definition 5.5]{M-O20} we introduced cubic versions $\MM_1(n)_{\mathrm{cub}}$ and $\MM_0(n)_{\mathrm{cub}}$ of the moduli stacks $\MM(\Gamma_1(n))$ and $\MM(\Gamma_0(n))$. These come with a finite morphism to the moduli stack $\MM_{\mathrm{cub}}$ of cubic curves, where we allow arbitrary Weierstra{\ss} equations. We showed in \cite[Theorem 5.19]{M-O20} that $\MM_1(n)_{\mathrm{cub}} \simeq [M(\Gamma_1(n); \Z[\frac1n])/\G_m]$ for $n\geq 2$. In combination, we see that $\XX_{\tmf_1(n)} \simeq \MM_1(n)_{\mathrm{cub}}$ for $n\geq 2$. In the case $n=1$, the corresponding equivalence $\XX_{\tmf}\simeq \MM_{\mathrm{cub}}$ has a quite different character and was shown in \cite{Mathom}. Whether there are equivalences $\XX_{\tmf_0(n)} \simeq \MM_0(n)_{\mathrm{cub}}$ for a suitable definition of $\tmf_0(n)$ remains open to the knowledge of the author, even for $n=3$.  
\end{remark}


\subsection{The $C_2$-equivariant argument} \label{sec:C2argument}
All the stacks $\MM(\Gamma)$ come with an involution induced from postcomposing the level structure with the $[-1]$-automorphism of the elliptic curve. As explained in \cref{rem:C2actions}, this induces a $C_2$-action on $\Tmf(\Gamma)$. Our goal in this subsection is to define suitable $C_2$-spectra $\tmf(\Gamma)$ in the tame case. This will allow us to construct an $E_{\infty}$-ring spectrum $\tmf(\Gamma)$ also if there is just a tame subgroup $\Gamma'\subset \Gamma$ of index $2$ (see \cref{constr:Gamma'}).  

\begin{remark}\label{rem:C2actions}
The goal of this remark is to clarify the construction of the $C_2$-action on $\Tmf(\Gamma)$ sketched above. 

Denote the automorphism of $\MM(\Gamma)$ described above by $t$. As $t$ commutes with the forgetful map $\pr\colon \MM(\Gamma) \to \MM_{ell}$, this defines a $C_2$-action inside the slice category $(\mathrm{Stacks}/\MM_{ell})^{\acute{e}t, op}$ of stacks \'etale over $\MM_{ell}$. We will use a lax commutative triangle
\[
\begin{tikzcd}[column sep=scriptsize]
(\mathrm{Stacks}/\MM_{ell})^{\acute{e}t, op} \arrow[dr, "\OO^{\top}"'{name=U}] \arrow[rr, "N"]
& & (\mathrm{Stacks}/\MMb_{ell})^{\mathrm{log}-\acute{e}t, op} \arrow[Rightarrow, to=U, shorten = 0.8cm, yshift=0.1cm, xshift =0.1cm ] \arrow[dl, "\OO^{\top}"] \\
& \CAlg(\Sp) 
\end{tikzcd}
\]
Here, the diagonal arrows are the Goerss--Hopkins--Miller and Hill--Lawson sheaves of ring spectra. The horizontal arrow $N$ is a normalization construction (see e.g.\ \cite[Proposition 2.27]{H-L16}). The canonical map $\OO^{top}(N(U)) \to \OO^{top}(U)$ for $U \to \MM_{ell}$ \'etale comes from the fact that $U \subset N(U)$ is an open substack and the Hill--Lawson sheaf restricts to the Goerss--Hopkins--Miller sheaf.  

Applying the left diagonal arrow to $(\MM(\Gamma), t)$ gives a $C_2$-action on $\TMF(\Gamma)$. Doing the same with the composite of the right diagonal arrow and the horizontal arrow produces the $C_2$-action on $\Tmf(\Gamma)$. Moreover, we obtain a $C_2$-map $\Tmf(\Gamma) \to \TMF(\Gamma)$. 

As explained in \cite[Example 6.12]{MeiTopLevel}, the $C_2$-action induced by $t$ on $\TMF(\Gamma)$ is equivalent to the one induced by the $C_2$-action in $(\mathrm{Stacks}/\MM_{ell})^{\acute{e}t, op}$ given by $\id_{\MM(\Gamma)}$ on $\MM(\Gamma)$, but choosing the $[-1]$-isomorphism between the elliptic curves classified by $\pr$ and $\pr\id_{\MM(\Gamma)}$. This $C_2$-action induces multiplication by $-1$ on the pullback of $\omega$ to $\MM(\Gamma)$: Indeed, the $[-1]$-automorphism of an elliptic curve induces multiplication by $-1$ on the sheaf of differentials. Moreover, $\pr$ classifies precisely the pullback of the universal elliptic curve $\EE^{\mathrm{uni}}$ and $\omega$ is the restriction of $\Omega^1_{\EE^{\mathrm{uni}}/\MM_{ell}}$ to $\MM_{ell}$ along the zero section. 

Thus, if $\Gamma$ is tame, this implies that $C_2$ acts by $(-1)^k$ on $\pi_{2k}\TMF(\Gamma) \cong H^0(\MM(\Gamma);\omega^{\tensor k})$. Since $\pi_{2k}\Tmf(\Gamma)$ injects in the tame case into $\pi_{2k}\TMF(\Gamma)$, the same is true for $\pi_{2k}\Tmf(\Gamma)$. 
    
Note that the action $t$ can be trivial, e.g.\ for $\Gamma = \Gamma_0(n)$ or $\Gamma(2)$. This forces $\pi_{2k}\Tmf(\Gamma) = 0$ for $k$ odd in these cases (as $t$ acts both by $1$ and $-1$ and the groups are torsionfree). This corresponds to the classical fact that there are no modular forms of odd weight if $-\id$ is in $\Gamma$. 
\end{remark}

In the following we will use standard notation from equivariant homotopy theory. In particular, we denote for an inner product space $V$ with $G$-action by $S(V)$ the unit sphere and by $S^V$ the $1$-point compactification as $G$-spaces. We denote by $a = a_{\sigma}\colon S^0 \to S^{\sigma}$ the inclusion for $\sigma$ the real sign representation of $C_2$.   

The Hopf map defines a $C_2$-map $\overline{\eta}\colon S(\C^2) \to S^{\C}$, where $C_2$ acts on $\C$ via complex conjugation. This stabilizes to an element in  $\pi_{\sigma}^{C_2}\S$, which restricts to $\eta \in \pi_1^e\S$. 

\begin{lemma}
The homotopy groups  $\pi_{\sigma}^{C_2}(\S)$ and $\pi_{-\sigma}^{C_2}\S$ are infinite cyclic and generated by $\overline{\eta}$ and $a$, respectively. 
\end{lemma}
\begin{proof}
For $\pi_{\sigma}^{C_2}(\S)$, this is proven as formula (8.1) in \cite{ArakiIriye}. (Note that they use the notation $\pi^s_{p,q}$ for our $\pi^{C_2}_{p\sigma +q}(\S)$.) Proposition 7.0 in op.\ cit.\ implies that the homomorphism $\pi_{-\sigma}^{C_2}\S \to \pi_0\S$, taking a map $\S \to \Sigma^{\sigma}\S$ to its geometric fixed points, is an isomorphism. Taking fixed points of the map $a$ clearly gives the identity map $S^0\to S^0$, which yields the result.
\end{proof}

In the following, we denote by $\tau_{\leq i}$ the slice coconnective cover, by $\tau_{\geq i}$ the slice connective cover and by $\tau_i = \tau_{\geq i}\tau_{\leq i}$ the $i$-th slice for $C_2$-spectra. We refer to \cite{HHR} for background about the slice filtration. We denote by $H\underline{\Z}$ the $C_2$-Eilenberg--MacLane spectrum for the constant Mackey functor $\underline{\Z}$. 

\begin{lemma}\label{lem:sliceHZ}
We have an equivalence $\tau_{\leq 1}C\overline{\eta} \simeq H\underline{\Z}$.
\end{lemma}
\begin{proof}
It suffices to show that the first slice of $C\overline{\eta}$ is null and the zeroth slice is $H\underline{\Z}$. As shown in \cite{HHR} and summarized in \cite[Section 2.4]{H-M17}, this is implied by the calculations $\underline{\pi}_0C\overline{\eta} \cong \underline{\Z}$ and $\underline{\pi}_{\sigma}C\overline{\eta} = 0$. These follows easily by the long exact sequence arising from the cofiber sequence 
\[S^{\sigma} \xrightarrow{\overline{\eta}} S^0 \to C\overline{\eta}\] 
and the computations of $\pi_{-\sigma}^{C_2}\S$, $\pi_0^{C_2}\S$ and $\pi_{\sigma}^{C_2}(\S)$ above, using also that $\pi_{-1}^{C_2}S^{\sigma} = 0$. 
\end{proof}

The following lemma is a $C_2$-slice analogue of Lewis's equivariant Hurewicz theorem \cite[Theorem 2.1]{LewisHurewicz}. Recall that a $C_2$-spectrum is \emph{$k$-slice connected} iff $\tau_{\leq k}X = 0$. 
\begin{lemma}\label{lem:sliceHurewicz}
A connective $C_2$-spectrum $X$ is $k$-slice connected iff $H\underline{\Z} \tensor X$ is $k$-slice connected.
\end{lemma}
Spelled out, the latter condition is equivalent to $\underline{H}_V(X;\underline{\Z}) = \underline{\pi}_V(H\underline{\Z}\tensor X)  = 0$ for all $C_2$-representations $V$ of the form $i\rho$ or $i\rho -1$ with $|V|\leq k$. 
\begin{proof}
If $X$ is $k$-slice connected, the same is true for $H\underline{\Z}\tensor X$. For the converse, assume that $\underline{H}_V^{C_2}(X;\underline{\Z}) = 0$ for all $C_2$-representations $V$ of the form $i\rho$ or $i\rho -1$ with $|V|\leq k$. By induction on $k$, we can assume that $X$ is $(k-1)$-slice connected and we need to show that $\tau_kX = 0$ to deduce that $X$ is indeed $k$-slice connected. Let $W$ be $\frac{k}2\rho$ if $k$ is even and $\frac{k+1}2\rho -1$ if $k$ is odd. As $\tau_{\geq k+1}X$ and its suspension are $k$-slice connected, the direction discused above shows $\underline{H}_W(\tau_{\geq k+1}X; \underline{\Z}) = \underline{H}_W(\Sigma\tau_{\geq k+1}X; \underline{\Z}) = 0$. Thus,
\[0 = \underline{H}_W(X;\underline{\Z}) \to \underline{H}_W(\tau_kX; \underline{\Z})\]
is an isomorphism. As summarized in \cite[Section 2.4]{H-M17}, the slice $\tau_kX$ is of the form $S^W \tensor HM$ for some Mackey functor $M$ and we deduce that $\underline{H}_0(HM; \underline{\Z}) \cong \underline{H}_W(\tau_kX; \underline{\Z}) = 0$. 

We know that $\tau_0\S = H\underline{\Z}$. As $HM$ is (slice) connective, a similar argument to before shows that 
\[M \cong \underline{\pi}_0(\S\tensor HM) \cong \underline{\pi}_0(H\underline{\Z} \tensor HM)  = \underline{H}_0(HM; \underline{\Z}) = 0.\]
Thus, $\tau_kX = 0$ as was to be shown. 
\end{proof}

For an element $x \in \pi_k^{C_2}\S$, we can define a (naive) $C_2$-equivariant $E_\infty$-cone $C^{\S}(x)$ as in the non-equivariant situation in the preceding subsection. The arguments for the following two results are quite analogous to those of the preceding section, so we allow ourselves to be brief. 

\begin{lemma}
The map $C(\overline{\eta}) \to C^{\S}(\overline{\eta})$ is slice-$2$-connected.
\end{lemma}
\begin{proof}
By \cref{lem:sliceHurewicz} it suffices to check that 
$C(\overline{\eta}) \tensor H\underline{\Z} \to C^{\S}(\overline{\eta}) \tensor H\underline{\Z}$ is slice-$2$-connected. Since $\underline{\pi}_{\sigma}H\underline{\Z} = 0$ and thus $\overline{\eta}$ becomes zero in $H\underline{\Z}$, this agrees with 
\[H\underline{\Z} \oplus \Sigma^{\rho}H\underline{\Z} \to C^{H\underline{\Z}}(\Sigma H\underline{\Z}) \simeq \P^{H\underline{\Z}}{\Sigma^{\rho}H\underline{\Z}} \simeq H\underline{\Z} \oplus \Sigma^{\rho}H\underline{\Z} \oplus (\Sigma^{2\rho}H\underline{\Z})_{hC_2}\oplus \cdots \]
Analogously to \cref{lem:splitinjective}, the map is split injective and thus indeed slice-$2$-connected (even slice-$3$-connected).
\end{proof}

Together with \cref{lem:sliceHZ} this implies that $\tau_{\leq 1} C^{\S}(\overline{\eta}) \simeq H\underline{\Z}$. To deduce the analogue of \cref{prop:killingpi1}, we will need one more categorical result. 

\begin{lemma}
Let $G$ be a finite group and $\Sp_G$ be the $\infty$-category of $G$-spectra. Denote by $\Sp_G^{\geq 0}$ the full subcategory of connective $G$-spectra and by $\Sp^{[0,k]}_G$ that of connective and slice-$k$-truncated $G$-spectra. Then the inclusion $\CAlg(\Sp^{[0,k]}_G) \to \CAlg(\Sp_G)$ admits for every $k\geq 0$ a left adjoint, which agrees on the level of underlying $G$-spectra with the slice truncation $\tau_{\leq k}$.
\end{lemma}
\begin{proof}
Connective $G$-spectra form a presentable $\infty$-category with compact generators the $\Sigma^\infty G/H_+$. We obtain $\Sp^{[0,k]}_G$ by localizing $\Sp_G^{\geq 0}$ at the collection $S$ of maps $C \to 0$ for $C$ a slice cell of dimension greater than $k$. By \cite[Proposition 5.5.4.15]{HTT}, $\Sp^{[0,k]}_G$ is presentable again. 

If $X$ is connective and $Y\geq k+1$ in the slice filtration, then $X \tensor Y\geq k+1$ by \cite[Proposition 4.26]{HHR}. Thus, $\tau_{\leq k}$ is compatible with $\tensor$ in the following sense: if $X \to Y$ in $\Sp_G^{\geq 0}$ induces an equivalence $\tau_{\leq k}X \to \tau_{\leq k}Y$, then $\tau_{\leq k}(X\tensor Z) \to \tau_{\leq k}(Y \tensor Z)$ is an equivalence for every $Z\in \Sp_G^{\geq 0}$. By \cite[Proposition 2.2.1.9]{HA}, $\Sp^{[0,k]}_G$ inherits the structure of a symmetric monoidal $\infty$-category from $\Sp_G^{\geq 0}$ and $\tau_{\leq k}$ is strong symmetric monoidal, while the inclusion $\Sp^{[0,k]}_G \to \Sp_G^{\geq 0}$ is lax symmetric monoidal. The same proposition gives that the resulting maps $\left(\Sp^{[0,k]}_G\right)^{\tensor} \to \left(\Sp_G^{\geq 0}\right)^{\tensor}$ and $\left(\Sp^{\geq 0}_G\right)^{\tensor} \to \left(\Sp_G^{[0,k]}\right)^{\tensor}$ of $\infty$-operads are adjoint. Since commutative algebras in such an $\infty$-operad $\CC^{\tensor}$ are defined as sections of $\CC^{\tensor} \to \mathrm{NFin}_*$ as maps of operads, we see that the resulting maps between $\CAlg(\Sp^{[0,k]}_G)$ and $\CAlg(\Sp_G^{\geq 0})$ are indeed adjoint. Here, we use the characterization of an adjunction given by \cite{RiehlVerityAdjunction}, namely the existence of a unit and counit, satisfying the triangle identities up to homotopy. 
\end{proof}

\begin{prop}\label{prop:Killingpi1C2}
Let $R$ be a connective $E_\infty$-ring $C_2$-spectrum with $\underline{\pi}_0^{C_2} = \Z_S$ being a localization of $\Z$ and $\overline{\eta} = 0 \in \pi_{\sigma}^{C_2}R$. Then there is an $E_\infty$-ring $C_2$-spectrum $R'$ with an $E_\infty$-map $R'\to R$ inducing an equivalence on slices in degree $0$ and degrees at least $2$ and such that $\tau_1R' = 0$. Moreover, for every other $R'' \to R$ with these properties, there is an equivalence $R''\to R'$ of $E_{\infty}$-ring $C_2$-spectra over $R$. 
\end{prop} 
\begin{proof}
Since $\overline{\eta}$ is zero in $R$, we obtain a map $C^{\S}(\overline{\eta})_S \to R \to \tau_{\leq 1}R$ of $E_{\infty}$-ring $C_2$-spectra, which factors over $H\underline{\Z}_S = \tau_{\leq 1}C^{\mathbb{S}}(\overline{\eta})_S$. Define now $R'$ via the homotopy pullback square
\[\xymatrix{R'\ar[r]\ar[d] & H\underline{\Z}_S \ar[d] \\
R \ar[r] & \tau_{\leq 1}R} 
\]
The proof of uniqueness is analogous to \cref{prop:killingpi1}. 
\end{proof}

To formulate the consequences for $\tmf(\Gamma)$, we want to recall from \cite{H-M17} that a $C_2$-spectrum $E$ is strongly even if its odd slices vanish and its even slices are of the form $S^{k\rho}\tensor H\underline{A}$ or, equivalently, if $\underline{\pi}_{k\rho}E$ is constant and $\underline{\pi}_{k\rho-1}E =0$.

\begin{thm}\label{prop:tmf1(n)C2}
	For every set of primes $S$ and every congruence subgroup $\Gamma_1(n)\subset \Gamma \subset \Gamma_0(n)$ that is tame with respect to $\Z_S$, we can define a strongly even connective $E_{\infty}$-ring $C_2$-spectrum $\tmf(\Gamma)_S$ with an $E_{\infty}$-ring $C_2$-map $\tmf(\Gamma)_S \to \Tmf(\Gamma)_S$ that identifies the underlying homotopy groups of the source with $M(\Gamma; \Z_S)$. 
\end{thm}
\begin{proof}
	Equip $\Tmf(\Gamma) \simeq \Tmf(\Gamma)^{E(C_2)_+}$ with the cofree structure of a $C_2$-spectrum. We claim that
	\begin{enumerate}
	    \item  $\overline{\eta} \in \pi_{\sigma}^{C_2}\tau_{\geq 0}\Tmf(\Gamma)$ is zero,
	    \item the only odd slice of $\tau_{\geq 0}\Tmf(\Gamma)$ is $\tau_1$, and
	    \item the even slices of $\Tmf(\Gamma)$ are of the form $S^{k\rho}\tensor H\underline{A}$.    
	\end{enumerate}
	Given these claims, applying \cref{prop:Killingpi1C2} to $R = \tau_{\geq 0}\Tmf(\Gamma)$ yields the $C_2$-spectrum $R' = \tmf(\Gamma)$ with the required properties: the first claim implies that we can apply \cref{prop:Killingpi1C2}, while the other two ensure that $\tmf(\Gamma)$ is strongly even. 
	
	For proving the claims, we will distinguish the (overlapping) cases that $\frac12\in \Z_S$ and that $\Gamma$ is tame for $\Z_{(2)}$. 
	
	For the first claim, note that $\pi_{\sigma}^{C_2}\tau_{\geq 0}\Tmf(\Gamma) \cong \pi_{\sigma}^{C_2}\Tmf(\Gamma)$ (e.g.\ since $S^{\sigma}$ is a slice cell). The restriction map $\pi_{\sigma}^{C_2}\Tmf(\Gamma) \to \pi_1\Tmf(\Gamma)$ is an injection: if $\frac12\in \Z_S$, this follows from the homotopy fixed points spectral sequence; else, use the line after (6.15) in \cite{MeiTopLevel}. Since $\overline{\eta}$ restricts to $\eta \in \pi_1\Tmf(\Gamma)$, \cref{lem:eta} implies thus the vanishing of $\overline{\eta}$. 
	
	If $\Gamma$ is tame for $\Z_{(2)}$, \cite[Theorem 6.16]{MeiTopLevel} yields the last two claims. If $\frac12\in S$, we obtain $\pi_{k\rho-1}^{C_2}\Tmf(\Gamma) =0$ by the homotopy fixed point spectral sequence since $\pi_{2k-1}\Tmf(\Gamma) = 0$ for $k>1$ by \cite[Section 2.1]{MeiTopLevel}. This yields the second claim by \cite[Proposition 2.9]{H-M17}. For the third claim it is enough to show that $\underline{\pi}_{k\rho}\Tmf(\Gamma)$ are constant Mackey functors (see \cite[Proposition 2.13]{H-M17}). This follows again from the homotopy fixed point spectral sequence and the fact that the $C_2$-action on $\pi_{k\rho}\Tmf(\Gamma)$ is trivial: indeed, $C_2$ acts by $(-1)^k$ on $\pi_{2k}\Tmf(\Gamma)$ (see \cref{rem:C2actions}) and the presence of $k\sigma$ twists the action by the same sign. 
\end{proof}

\begin{remark}
    The case that $\Gamma = \Gamma_0(n)$ is not excluded in the previous theorem, but one easily checks that $\Gamma_0(n)$ can only be tame if $\frac12\in\Z_S$. In this case, we obtain simply the cofree $C_2$-spectrum of $\tmf_0(n)_S$ with the trivial action. 
\end{remark}

\begin{construction}\label{constr:Gamma'}
    Given $\Gamma' \subset \Gamma \subset \Gamma_0(n)$ with $\Gamma'$ tame with respect to $\Z_S$ and $\Gamma/\Gamma' \cong C_2$, we can extend our previous definition by defining $\tmf(\Gamma)_S$ as $\tmf(\Gamma')_S^{C_2}$ (so e.g.\ $\tmf_0(3) = \tmf_1(3)^{C_2}$ as in \cite{H-M17}). If $\Gamma$ itself is already tame, then $\frac12\in\Z_S$. One then easily computes (e.g.\ with the slice spectral sequence) that $\pi_*\tmf(\Gamma')_S^{C_2} \cong \pi_*\tmf(\Gamma)_S$ and one can use the uniqueness part of \cref{thm:tmfGamma} to identify our new definition with the previous one. 
\end{construction}

\begin{remark}\label{rem:fiber}
    In the setting of the previous construction, the map $\tmf(\Gamma) \to \tau_{\geq 0}\Tmf(\Gamma)$ is an isomorphism in $\pi_*$ for $\ast \geq 2$ even if $\Gamma$ is not tame. Indeed, the cofiber of $\tmf(\Gamma') \to \tau_{\geq 0}\Tmf(\Gamma')$ is the target's first slice and thus by \cite[Theorem 6.16]{MeiTopLevel} equivalent to $\Sigma^{\sigma} HM$ for $M$ being the constant Mackey functor on $H^1(\MMb(\Gamma')_S; \omega) \cong \pi_1\Tmf(\Gamma')$. We directly observe that the non-equivariant homotopy groups of $\Sigma^{\sigma} HM$ vanish in degrees at least $2$. Moreover the cofiber sequence $(C_2)_+ \to S^0 \to S^{\sigma}$ induces a long exact sequence 
\[\pi^{e}_{k}HM \to \pi_k^{C_2}HM \to \pi_{k-\sigma}^{C_2}HM \to \pi^e_{k-1} HM \xrightarrow{\mathrm{tr}} \pi^{C_2}_{k-1}HM,\]
which implies that $\pi_k^{C_2}\Sigma^{\sigma}HM = \pi_{k-\sigma}^{C_2}HM = 0$ for $k\geq 2$ and actually also for $k=1$ if $\mathrm{tr}$ is injective, i.e.\ if $\pi_1\Tmf(\Gamma')$ has no $2$-torsion. Thus, $\tmf(\Gamma) \to \tau_{\geq 0}\Tmf(\Gamma)$ is indeed an isomorphism in $\pi_*$ for $\ast \geq 2$ and even for $\ast = 1$ if $\pi_1\Tmf(\Gamma')$ has no $2$-torsion.
\end{remark}

\section{Realization of Hirzebruch's level-$n$ genus}\label{sec:genus}
In the previous section we have defined ring spectra $\tmf_1(n) = \tmf(\Gamma_1(n))$. 
The spectra $\tmf_1(n)$ are even for $n\geq 2$ and thus complex orientable. We want to show that there is a complex orientation for $\tmf_1(n)$ such that the corresponding map 
$$MU_{2*} \to \tmf_1(n)_{2*} \cong M(\Gamma_1(n);\Z[\frac1n])$$
agrees with the level-$n$ genus introduced by Hirzebruch \cite{HirElliptic} and Witten \cite{WittenIndex} and studied e.g.\ in \cite{Krichever}, \cite{Fra92}, \cite{Herrera} and \cite{WWYElliptic}. We we recall its definition below. For this purpose it will be convenient to use algebro-geometric language, for which we recall first the following set of definitions.   

\begin{defi}
	A \emph{formal group} over a base scheme $S$ is a Zariski sheaf $F\colon \Sch_S^{op} \to \Ab$ that Zariski locally on an affine open $U=\Spec R \subset S$ is isomorphic to $\Spf R\llbracket t \rrbracket$. The $R$-modules $R\llbracket t\rrbracket$ glue to the \emph{structure sheaf} $\OO_F$ on $S$ and the $R$-modules $(R\llbracket t\rrbracket/t)\cdot dt$ glue to the line bundle $\omega_{F/S}$.\footnote{If $p\colon C \to S$ is a (generalized) elliptic curve and $F$ is the formal completion of $\EE$, this agrees with $\omega_{C/S} = p_*\Omega^1_{C/S}$.} An \emph{invariant differential} of a formal group $F$ is a trivialization of $\omega_{F/S}$. A \emph{coordinate} is a section $s$ of $\OO_F$ that is of the form $a_0t + a_1t^2+\cdots$ with $a_0\in R^\times$ for every local trivialization $F|_{\Spec R} \cong \Spf R\llbracket x \rrbracket$. 
\end{defi}
\begin{remark}
    There are different ways to state the definition of a formal group, e.g.\ as an abelian group object in one-dimensional formal Lie varieties (see \cite[Definitions 1.29 and 2.2]{GoerssModuliFG}). To compare them, note that our formal groups are automatically fpqc sheaves since $\Spf R\llbracket t \rrbracket$ is an fpqc sheaf. On the other hand, a trivialization of the sheaf of differentials of a one-dimensional formal Lie variety over $\Spec R$ determines an equivalence to $\Spf R\llbracket t \rrbracket$ and such trivializations exist Zariski locally. 
\end{remark}
We note that the differential $ds$ of a coordinate $s$ of a formal group $F$ is an invariant differential of $F$, sending $a_0t + a_1t^2+\cdots$ to $a_0dt$ locally. If $S = \Spec R$, a coordinate of $F$ is equivalent datum to an isomorphism $F \cong \Spf R\llbracket s\rrbracket$. 

Recall that given an arbitrary even ring spectrum $E$, a complex orientation is an element in $\widetilde{E}^2(\CP^{\infty})$ restricting to $1 \in \widetilde{E}^2(\CP^1)$ after a homeomorphism $\CP^1\cong S^2$ is chosen. The formal spectrum $\Spf E^{2*}(\CP^{\infty})$ is a formal group over $\Spec E^{2*}(\pt)$ and the line bundle $\omega$ corresponds to $\widetilde{E}^*(\CP^1)$; it thus comes with a canonical invariant differential corresponding to $1 \in \widetilde{E}^2(\CP^1)$. A complex orientation is thus a coordinate of $\Spf E^{2*}(\CP^{\infty})$ in degree $\ast = 1$ whose differential is the canonical invariant differential. 

We want to apply this to $E = \tmf_1(n)$ for $n\geq 2$. Essentially by construction, the maps $\pi_{2*}\tmf_1(n) \to \pi_{2*}\Tmf_1(n) \to H^0(\MMb_1(n);\omega_{\overline{\CC}/\MMb_1(n)}^{\tensor*})$ are isomorphisms, where $\overline{\CC}$ is the universal generalized elliptic curve over $\MMb_1(n)$. For convenience, let $\MMb_1^1(n)$ be the relative spectrum $\underline{\Spec}_{\MMb_1(n)}\left( \bigoplus \omega_{\overline{\CC}/\MMb_1(n)}^{\tensor*}\right)$, which is the total space of the $\G_m$-torsor associated with $\omega_{\overline{\CC}/\MMb_1(n)}$, i.e.\ classifies generalized elliptic curves with a point of exact order $n$ and an invariant differential. The resulting morphism 
\[\MMb_1^1(n) \to \Spec H^0(\MMb_1^1; \OO_{\MMb_1^1(n)}) \cong \Spec H^0(\MMb_1; \omega_{\overline{\CC}/\MMb_1(n)}^{\tensor*}) \cong \Spec \pi_{2*}\tmf_1(n)\]
is an open immersion whose image is covered by the non-vanishing loci of $c_4$ and $\Delta$ \cite[Proposition 3.5]{M-O20}. We denote by $\CC$ the pullback of $\overline{\CC}$ to $\MMb_1^1(n)$. Since $\tmf_1(n)[c_4]^{-1} \simeq \Tmf_1(n)[c_4^{-1}]$ and $\tmf_1(n)[\Delta^{-1}] \simeq \Tmf_1(n)[\Delta^{-1}]$ are elliptic cohomology theories, their formal groups are identified with the restrictions of $\widehat{\CC}$ to the non-vanishing loci of $c_4$ and $\Delta$, respectively, and as a result $\widehat{\CC}$ becomes identified with the restriction of $\Spf \tmf_1(n)^{2*}(\CP^{\infty})$ to $\MMb_1^1(n)$. As $\MMb_1^1(n) \subset \Spec \pi_{2*}\tmf_1(n)$ induces an isomorphism on global sections of the structure sheaf, coordinates on $\Spf \tmf_1(n)^{2*}(\CP^{\infty})$ are in bijection with those on $\widehat{\CC}$ and one checks that the canonical invariant differential on the former corresponds to the canonical invariant differential on the latter. Summarizing we obtain:

\begin{lemma}\label{lem:cxor}
	Complex orientations $MU \to \tmf_1(n)$ are in bijection with coordinates of $\widehat{\CC}$, which are homogeneous of degree $1$ and have the canonical invariant differential as differential.
\end{lemma}

 The Hirzebruch genus relies on a specific such coordinate, which we will construct momentarily. Basically we will follow \cite[Chapter 7]{HBJ}, but present a more algebro-geometric approach and give an independent treatment. The key point is the existence of a certain meromorphic function on a cover of a given generalized elliptic curve. Recall to the purpose of constructing this function that every section $P$ into the smooth part of a generalized elliptic curve $C \to S$ is an effective Cartier divisor \cite[Lemma 1.2.2]{K-M85}, i.e.\ the kernel $\OO_C(-(P))$ of $\OO_C \to P_*\OO_S$ is a line bundle. Given any linear combination of sections $P_i$, we denote by $\OO_C(\sum_in_i(P_{i}))$ the corresponding tensor product of line bundles. 

\begin{lemma}
	Let $n\geq 2$ and $S$ be a $\Z[\frac1n]$-scheme. Furthermore let $C/S$ be a generalized elliptic curve with $0$-section $e\colon S \to C$ and a chosen point $P\colon S \to C$ of exact order $n$ in the smooth locus.  
	\begin{enumerate}[(a)]
		\item The pullback of $e^*\OO_{C}((P) - (e))$ to $S$ is canonically isomorphic to $\omega_{C/S} = e^*\Omega^1_{C/S}$.
		\item Let $\lambda$ be an invariant differential on $C$. Then there exists a unique meromorphic function $h$ on $C$ with an $n$-fold zero at $e$ and an $n$-fold pole at $P$ as only pole whose restriction along $e$ coincides with $\lambda^n$ under the identification of the previous part. 
		\item There exists a degree-$n$ \'etale cover $q\colon C' \to C$ by a generalized elliptic curve and a meromorphic function $f$ on $C'$ with $f^n = q^*h$. 
	\end{enumerate}
\end{lemma}
\begin{proof}
	For the proof of (a), note that $\OO_C(-(e))$ is the ideal sheaf associated to the closed immersion $e$ and the pullback $e^*\OO_{C}((P) - (e))$ coincides with $\OO_C(-(e))/\OO_C(-(e))^2$ viewed as an $\OO_S$-module. Indeed: We can cover $S$ by opens of the form $U\cap S$, where $U\cong \Spec R$ is an affine open in $C$ not intersecting the image of $P$. The section $e$ corresponds to an element $s\in R$ and $U\cap S \cong \Spec S/s$. Then $e^*\OO_C((P)-(e))(U\cap S)$ is the $S/s$-module $sS \tensor_{S} S/s$, which is canonically isomorphic to the $S/s$-module $sS/s^2S$. 
	
	E.g.\ by \cite[Proposition II.8.12]{Har77} we obtain a canonical surjective map \[\OO_C(-(e))/\OO_C(-(e))^2 \to e^*\Omega^1_{C/S} = \omega_{C/S}\]
	between line bundles, which is hence an isomorphism. 
	
	For part (b), consider the line bundle $\OO_{C}(n(P) - n(e))$. Note that $n\cdot P - n\cdot e = e$ as points on $C$. By \cite[Theorem 2.1.2]{K-M85} in the case that $C$ is an elliptic curve and by \cite[Proposition II.2.7]{D-R73} for generalized elliptic curves, we deduce that $\OO_{C}(n(P) - n(e))$ is the pullback of a line bundle $\LL$ on $S$. By part (a), $\LL= e^*p^*\LL = \omega_{C/S}^{\tensor n}$. By \cite[Proposition II.1.6]{D-R73}, we see that the canonical map 
	\[\omega_{C/S}^{\tensor n} \to p_*p^*\omega_{C/S}^{\tensor n} \cong p_*\OO_{C}(n(P) - n(e))\]
	is an isomorphism. Thus
	\[\Gamma(\OO_{C}(n(P) - n(e))) \cong \Gamma(\omega_{C/S}^{\tensor n}), \]
	where the isomorphism can be identified with the pullback along $e$. Thus, there is a unique section $h$ of $\OO_{C}(n(P) - n(e))$ whose image is $\lambda^n$.  
	
	For part (c), consider the $\mu_n$-torsor $q\colon C' \to C$ associated with the problem of extracting an $n$-th root out of $q^*h$ as a section of $q^*\OO_C((P)-(e))$ (i.e.\ the $\mu_n$-torsor associated with the pair ($h$, $\OO_C((P)-(e))$) in the sense of \cite[p.\ 125]{Mil80}). By construction, the required root $f$ exists on $C'$. By \cite[Proposition II.1.17]{D-R73}, $C'$ has the structure of a generalized elliptic curve provided that we can lift $e$ to $C'$ and $C' \to S$ has geometrically connected fibers. For the first point, it suffices to provide a section of $C' \times_C S \to S$, i.e.\ to provide an $n$-th root of $e^*h$. Under the identification of part (a), this is provided by $\lambda$. For the second point, we assume that $S = \Spec K$ with $K$ algebraically closed of characteristic not dividing $n$ and that $C'$ is not connected. The stabilizer of a component $C_0'$ must be of the form $\mu_m$ with $m<n$ and thus $C' \cong C_0' \times_{\mu_m} \mu_n$. The $\mu_m$-torsor $C_0'$ is hence associated with a pair $(g, \OO_C((P)-(e)))$ with $g^{n/m} = h$. The section $g$ provides a trivialization of $\OO_C(m(P) - m(e))$. This implies $m\cdot P = e$ on $C'$ \cite[Corollaire II.2.4]{D-R73}, in contradiction with $P$ being of exact order $n$. 
\end{proof}

\begin{construction}
	Let $\CC$ be the universal generalized elliptic curve with a point of exact order $n$ over $\MMb_1^1(n)$. It comes by definition with a canonical invariant differential $\lambda$. From the preceding lemma, we obtain an $n$-fold \'etale cover $q\colon \CC'\to \CC$ together with a meromorphic function $f$ on $\CC'$ whose pullback along a lift of $e$ agrees with $\lambda$. This function $f$ provides a coordinate for $\widehat{\CC'} \cong \widehat{\CC}$. Moreover note that $f$ is uniquely determined by the requirements in the lemma because $\CC'$ is irreducible (since $\MMb_1^1(n)$ is irreducible and the locus of smoothness of $\CC'$ in it is dense) and thus every other $n$-th root of $h$ would have to differ by a root of unity, resulting in a different pullback to $\MMb_1^1(n)$.
	
	Pulling the  orientation induced from $f$ back along a map $\Spec \C \to \MMb_1(n)$ classifying $(\C/\Lambda, \frac1n, dz)$ results exactly in the coordinate and orientation chosen in \cite{HBJ}. 
\end{construction}

\begin{thm}\label{thm:HirzebruchRealization}
	For every $n\geq 2$, there is a unique complex orientation of $MU \to \tmf_1(n)$ realizing on homotopy groups the Hirzebruch genus. Moreover, this can be uniquely refined to a morphism $MU_{\R} \to \tmf_1(n)$ of $C_2$-ring spectra.
\end{thm}
\begin{proof}
	The first part follows from \cref{lem:cxor} as the Hirzebruch genus is given by a coordinate on the formal group associated with the universal generalized elliptic curve on $\MMb_1^1(n)$. For the second point, we recall from \cite[Theorem 2.25]{H-K01} that $C_2$-ring morphisms $MU_{\R} \to \tmf_1(n)$ are in bijection with Real orientations of $\tmf_1(n)$, i.e.\ a lift of a complex orientation to a class $\tmf_1(n)^{\rho}_{C_2}(\CP^{\infty})$. As $\CP^{\infty}$ can be built by cells in dimensions $k\rho$, the strong-evenness of $\tmf_1(n)$ from \cref{prop:tmf1(n)C2} implies that the forgetful map 
	\[ \tmf_1(n)^{\rho}_{C_2}(\CP^{\infty}) \to \tmf_1(n)^2(\CP^{\infty}) \]
	is an isomorphism; thus every complex orientation of $\tmf_1(n)$ refines to a unique Real orientation. 
\end{proof}

\begin{rmk}
	We remark that in \cite{Fra92}, Franke already gave a related but different algebro-geometric treatment of the Hirzebruch genus. 
\end{rmk}

\begin{rmk}
    After the first version of this article became available, Senger has shown in \cite{SengerObstruction} that the map $MU \to \tmf_1(n)$ actually refines to one of $E_{\infty}$-ring spectra. He also gives a reformulation of our treatment above in terms of $\Theta^1$-structures.  
\end{rmk}

\section{Compactness, formality and faithfulness of $\tmf(\Gamma)$}\label{sec:compact}
Given a (tame) congruence subgroup of level $n$, we will show that $\tmf(\Gamma)$ is a faithful and perfect $\tmf[\frac1n]$-module. In contrast, for example $\tmf_1(3)$ will not be a perfect $\tmf_0(3)$-module, even rationally. The latter result relies on $\tmf_0(3)_{\Q}$ being formal (i.e.\ multiplicatively a graded Eilenberg--MacLane spectrum), a result we prove in greater generality in a subsection on its own. 

\subsection{All $\tmf(\Gamma)$ are perfect}
Recall that for an $A_{\infty}$-ring spectrum $R$, a perfect $R$-module is a compact object in the $\infty$-category of left $R$-modules. Equivalently, the $\infty$-category of perfect $R$-modules is the smallest stable sub-$\infty$-category of all left $R$-modules that contains $R$ and is closed under retracts. 
The goal of this section is to show that the spectra $\tmf(\Gamma)$, in the cases we defined them, are perfect $\tmf[\frac1n]$-modules.  The key technical tool is the following proposition. 

\begin{prop}\label{prop:regnoet}
Let $R$ be an $A_\infty$-ring spectrum such that
\begin{enumerate}
\item $\pi_0R$ is regular noetherian,
\item all $\pi_nR$ are finitely generated $\pi_0R$-modules, and
\item $H\pi_0R$ is perfect as a $\tau_{\geq 0}R$-module. 
\end{enumerate}
Let furthermore $M$ be a perfect $R$-module. Then $\tau_{\geq k}M$ is a perfect $\tau_{\geq 0}R$-module for every $k\in\Z$.
\end{prop}
\begin{lemma}\label{lem:finitely}
With notation as in the statement of the proposition, let $X$ be a $\tau_{\geq 0}R$-module with only finitely many non-trivial homotopy groups, all finitely generated over $\pi_0R$. Then $X$ is a perfect $\tau_{\geq 0}R$-module.
\end{lemma}
\begin{proof}
By induction, we can reduce to the case that $\pi_*X$ is concentrated in a single degree $n$. Then $X = H\pi_nX$ acquires the structure of a $H\pi_0R$-module and it is perfect as such because $\pi_0R$ is regular noetherian and $\pi_nX$ is finitely generated. As $H\pi_0R$ is perfect over $\tau_{\geq 0}R$, the same is thus true for $X$. 
\end{proof}

\begin{proof}[Proof of proposition]
	Let $M$ be a perfect $R$-module. As the truth of the conclusion of the proposition is clearly preserved under retracts in $M$ and also clear for $M=0$, we can assume by induction that we have a cofiber sequence
	\[\Sigma^lR \to N \to M \to \Sigma^{l+1}R\]
	where $\tau_{\geq k}N$ is a perfect $\tau_{\geq 0}R$-module for all $k\in\Z$. Taking $\tau_{\geq l}$ on the first two objects gives a diagram
\[\xymatrix{
\Sigma^l\tau_{\geq 0}R\ar[d] \ar[r] & \ar[d] \tau_{\geq l}N \ar[r]&\ar[d] M' \ar[r]&\Sigma^{l+1}\tau_{\geq 0}R\ar[d] \\
\Sigma^l R \ar[r] & N \ar[r] & M\ar[r] & \Sigma^{l+1}R. }\]
of cofiber sequences. As $\tau_{\geq l}N$ is a perfect $\tau_{\geq 0}R$-module, so is $M'$. Clearly, $\tau_{\geq l+1}M' \simeq \tau_{\geq l+1}M$. As the fiber of $\tau_{\geq l+1}M' \to M'$ fulfills the conditions of the previous lemma, $\tau_{\geq l+1}M$ is perfect as a $\tau_{\geq 0}R$-module. 

For a general $k\in\Z$, we make a case distinction: Assume first that $k\geq l+1$. Then the fiber of $\tau_{\geq k} M \to \tau_{\geq l+1}M$ is perfect by the previous lemma, hence $\tau_{\geq k} M$ is perfect as well. If $k\leq l+1$, consider the fiber of $\tau_{\geq l+1} M \to \tau_{\geq k}M$ instead.
\end{proof}

To apply \cref{prop:regnoet} to topological modular forms, we need the following lemma. 
\begin{lemma}
For every $n\geq 1$, the $\tmf[\frac1n]$-module $H\pi_0\tmf[\frac1n] = H\Z[\frac1n]$ is perfect. 
\end{lemma}
\begin{proof}
If $2|n$, there is a $3$-cell complex $X$ such that $\tmf[\frac1n] \tensor X \simeq \tmf_1(2)[\frac1n]$ (see \cite[Theorem 4.13]{Mathom}). We have $\pi_*\tmf_1(2)[\frac1n] = \Z[\frac1n][b_2,b_4]$. Killing $b_2$ and $b_4$ gives $H\Z[\frac1n]$. Thus, $H\Z[\frac1n]$ is a perfect $\tmf_1(2)[\frac1n]$-module and hence also a perfect $\tmf[\frac1n]$-module.

If $3|n$, there is an $8$-cell complex $X$ such that $\tmf[\frac1n] \tensor X \simeq \tmf_1(3)[\frac1n]$ (see \cite[Theorem 4.10]{Mathom}). We have $\pi_*\tmf_1(3)[\frac1n] = \Z[\frac1n][a_1,a_3]$. Killing $a_1$ and $a_3$ gives $H\Z[\frac1n]$ and thus $H\Z[\frac1n]$ is also a perfect $\tmf[\frac1n]$-module in this case. 

For the general case, let $X_i$ be a collection of $\tmf[\frac1n]$-modules. Consider 
\[\Phi_k\colon \bigoplus_i \Hom_{\tmf[\frac1n]}\left(H\Z[\frac1n], X_i[\frac1k]\right) \to \Hom_{\tmf[\frac1n]}\left(H\Z[\frac1n], \bigoplus_i
 X_i[\frac1k]\right).\] 
If $k=2,3$ or $6$, then $\Phi_k$ is an equivalence by the previous results. As for every spectrum $X$, there is a cofiber sequence
\[\Sigma^{-1}X[\frac16] \to X \to X[\frac12]\oplus X[\frac13] \to X[\frac16] \]
there is a cofiber sequence of maps between mapping spectra 
\[\Sigma^{-1}\fib(\Phi_6) \to \fib(\Phi_1) \to \fib(\Phi_2\oplus \Phi_3) \to \fib(\Phi_6).\]
It follows that $\Phi_1$ is an equivalence as well and that $H\Z[\frac1n]$ is a perfect $\tmf[\frac1n]$-module. 
\end{proof}

\begin{thm}\label{thm:tmfperfect}
Let $\Gamma$ be a congruence subgroup of level $n$, which is tame or has a subgroup $\Gamma' \subset \Gamma$ of index $2$ with $\Gamma'$ tame. Then $\tmf(\Gamma)$ is a perfect  $\tmf[\frac1n]$-module. 

The same conclusion holds without the tameness hypothesis for any $\tmf[\frac1n]$-module $R$ with a map $R \to \tau_{\geq 0} \Tmf(\Gamma)$ whose fiber has finitely generated homotopy groups over $\Z[\frac1n]$, concentrated in finitely many degrees. 
\end{thm}
\begin{proof}
According to \cite[Proposition 2.12]{MeiTopLevel} the $\Tmf[\frac1n]$-module $\Tmf(\Gamma)$ is perfect. All $\pi_k\Tmf[\frac1n]$ are finitely generated $\Z[\frac1n]$-modules. Furthermore, $H\pi_0\tmf[\frac1n] = H\Z[\frac1n]$ is a perfect $\tmf[\frac1n]$-module by the previous lemma. This implies that $\tau_{\geq 0}\Tmf(\Gamma)$ is a perfect $\tmf[\frac1n]$-module by \cref{prop:regnoet}. 

For any $R$ as in the statement of the theorem, $R$ is thus perfect as well by \cref{lem:finitely}. To see that $\tmf(\Gamma)$ satisfies the hypotheses on $R$, note first that every $H^s(\MMb(\Gamma);\omega^{\tensor t})$ is a finitely generated $\Z[\frac1n]$-module for every $s$ and $t$ since $\MMb(\Gamma)$ is proper over $\Z[\frac1n]$. 
If $\Gamma$ is tame, the cofiber of $\tmf(\Gamma) \to \tau_{\geq 0}\Tmf(\Gamma)$ is by construction $H\pi_1\Tmf(\Gamma)$ and $\pi_1\Tmf(\Gamma)\cong H^1(\MMb(\Gamma); \omega)$. If  there is a tame subgroup $\Gamma'\subset \Gamma$ of index $2$, the cofiber $\tmf(\Gamma) \to \tau_{\geq 0}\Tmf(\Gamma)$ agrees with $\Sigma^{\sigma}HM$ for $M$ the constant Mackey functor on $H^1(\MMb(\Gamma'); \omega)$ by \cref{rem:fiber}. The exact sequence given in the same remark implies that the homotopy groups of $\Sigma^{\sigma}HM$ are concentrated in degrees $0$ and $1$ and are finitely generated $\Z[\frac1n]$-modules. 
\end{proof}

We recall from \cite{M-R99} that a connective $p$-complete spectrum $X$ is called an \emph{fp-spectrum} if $H_*(X;\F_p)$ is finitely presented as a comodule over the dual Steenrod algebra. They show in \cite[Proposition 3.2]{M-R99} that equivalently there is a finite spectrum $F$ with non-trivial $\F_p$-homology such that the total group $\pi_*(X\tensor F)$ is finite. The following proposition can be deduced from the known $\F_p$-(co)homology of $\tmf$ (see e.g.\ \cite[Section 21]{Rezk512}) and was already noted in \cite{M-R99} for $p=2$. We prefer to give a less computational proof though. 

\begin{prop}
	The $p$-completion of $\tmf$ is an fp-spectrum for all primes $p$. 
\end{prop}
\begin{proof}
	We implicitly $p$-localize. For $p\neq 3$, \cite[Theorem 4.10]{Mathom} implies the existence of a finite spectrum $W$ with non-trivial $\F_p$-homology such that $\tmf \tensor W \simeq \tmf_1(3)$. Choose a complex $V$ such that $BP_*V \cong BP_*/(p^{k_0}, v_1^{k_1}, v_2^{k_2})$ with $k_0, k_1$ and $k_2$ positive integers. As $\TMF_1(3)$ is Landweber exact, the sequence $p, v_1, v_2$ and hence the sequence $p^{k_0}, v_1^{k_1}, v_2^{k_2}$ is regular on $\pi_*\TMF_1(3)$. Since $\pi_*\tmf_1(3) = \Z_{(p)}[a_1, a_3]$ is an integral domain, the sequence is also regular on $\pi_*\tmf_1(3)$. Thus, 
	\[\pi_*\tmf\tensor W\tensor V \cong \pi_*\tmf_1(3)\tensor V \cong \pi_*\tmf_1(3)/(p^{k_0}, a_1^{k_1}, a_3^{k_2})\]
	 is a finitely generated $\Z/p^{k_0}$-algebra and of Krull dimension $0$. Hence it is of finite length as a $\Z/p^{k_0}$-module and thus finite. 
	 
	 Essentially the same argument works for $p=3$ if we choose instead a complex $W'$ with $\tmf \tensor W' \simeq \tmf_1(2)$ as in \cite[Theorem 4.13]{Mathom}. 
\end{proof}

\begin{cor}\label{cor:fp}
	The $p$-completion of $\tmf(\Gamma)$ for a congruence subgroup $\Gamma$ of level $n$ and $p$ not dividing $n$ is an fp-spectrum.
\end{cor}

For implications involving duality we refer to \cite{M-R99} and for an implication for the Hurewicz image in $H_*(\Omega^\infty\tmf(\Gamma); \F_p)$ to \cite[Theorem 1.7]{KuhnHurewicz}. 

\subsection{All $\tmf(\Gamma)_{\Q}$ are formal}
The goal of this section is to show that the $E_{\infty}$-rings $\tmf(\Gamma)_{\Q}$ are formal. While this statement is interesting in its own right, we also need it for further pursuing compactness questions in the following subsection. We begin with the following consequence of Goerss--Hopkins obstruction theory. 
\begin{prop}\label{prop:smooth}
	Let $A$ and $B$ be $E_{\infty}$-$H\Q$-algebras such that $\pi_*A$ is smooth as a $\Q$-algebra. Then 
	\[\pi_i\Map_{\CAlg}(A,B) \cong \begin{cases}
	\Hom_{\mathrm{grCRings}}(\pi_*A, \pi_*B) & \text{ if } i = 0 \\
	\Hom_{\pi_*A}(\Omega^1_{\pi_*A/\Q}, \pi_{*+i}B) & \text{ if }i>0,
	\end{cases}\]
	where for $\pi_i$ with $i>0$ a base point is chosen if a map $A \to B$ exists.
\end{prop}
\begin{proof}
	According to \cite[Section 4]{G-H04} or \cite[Section 6]{PPvK} with $E= H\Q$, there is an obstruction theory for lifting a morphism $\pi_*A \to \pi_*B$ to a morphism $A\to B$, where the obstructions lie in $\Ext^{n+1,n}_{\pi_*A}(\mathbb{L}^{E_{\infty}}_{\pi_*A/\Q}, \pi_*B)$, where $\mathbb{L}^{E_{\infty}}$ denotes the $E_{\infty}$-cotangent complex. As we are working rationally, this coincides with other forms of the cotangent complexes. In particular, we obtain from the smoothness of $\pi_*A$ that $\mathbb{L}^{E_{\infty}}_{\pi_*A/\Q}$ is isomorphic to $\Omega^1_{\pi_*A/\Q}$ concentrated in degree $0$, which again by smoothness is a projective $\pi_*A$-module. Thus the Ext-groups vanish and there is no obstruction to lifting a morphism $\pi_*A \to \pi_*B$ to a morphism $A\to B$. The same sources provide a spectral sequence computing $\pi_*\Map_{\CAlg}(A,B)$, which collapses by a similar Ext-calculation and gives the result. 
\end{proof}

\begin{prop}
	Let $\XX$ be a smooth Deligne--Mumford stacks over $\Q$ and $\OO$ an even-periodic sheaf of $E_{\infty}$-ring spectra on $\XX$ such that $\pi_0\OO\cong \OO_{\XX}$ and the $\pi_i\OO_{\XX}$ are quasi-coherent. Assume further that $H^{i+1}(\XX; \pi_i\OO) = 0$ for all even $i\geq 1$. Then $\OO$ is formal, i.e.\ equivalent to the (sheafification of the pre)sheaf $H\pi_*\OO$ of graded Eilenberg--MacLane spectra. 
\end{prop}
\begin{proof}
	Note first that $(\XX,\OO)$ actually defines a non-connective spectral Deligne--Mumford stack and in particular $\OO$ is hypercomplete (cf.\ e.g.\ \cite[Lemma B.2]{MeiTopLevel}). Set $\OO' = H\pi_*\OO$. Choosing an \'etale hypercover $U_{\bullet} \to \XX$ by affines, we can compute $\Map_{\CAlg_{\XX}}(\OO, \OO')$ as the totalization of the cosimplicial diagram $M^{\bullet} = \Map_{\CAlg}(\OO(U_{\bullet}), \OO'(U_{\bullet}))$. We observe using \cref{prop:smooth} that $\pi^0\pi_0M^{\bullet}$ agrees with the set of ring morphisms $\pi_*\OO \to \pi_*\OO'$, in which we can pick an isomorphism $f_0$. According to \cite[Sections 5.2, 2.4]{Bou89}, the vanishing of $\pi^{i+1}\pi_iM^{\bullet} \cong H^{i+1}(\XX, \pi_i\OO)$ for $i\geq 1$ suffices to lift $f_0$ to a multiplicative map $\OO\to \OO'$, which is automatically an equivalence. 
\end{proof}

\begin{cor}
	For all $\MMb(\Gamma)$ the rationalized Goerss--Hopkins--Miller--Hill--Lawson sheaf $\OO^{top}$ is formal. 
\end{cor}
\begin{proof}
	We can apply the previous proposition as $\MMb(\Gamma)_{\Q}$ has cohomological dimension $1$. (See e.g.\ \cite[Proposition 2.4(4)]{MeiDecMod}.)
\end{proof}

\begin{remark}
	In the original account of the construction of $\OO^{top}$ on $\MMb_{ell}$ in \cite{TMF}, $\OO^{top}_{\Q}$ is actually formal \emph{by construction}. Our argument shows that this choice was necessary, not only for $\MMb_{ell}$, but also for $\MMb(\Gamma)$. (The former was shown in a different manner already in \cite[Proposition 4.47]{HL13}.)
\end{remark}

\begin{prop}\label{prop:GammaFormal}
	Let $\Gamma$ be a congruence group. Then the $E_{\infty}$-rings $\tmf(\Gamma)_{\Q}$ are formal.
\end{prop}
\begin{proof}
	 Set $R = H(H^0(\MMb(\Gamma), \pi_*\OO^{top}_{\Q}))$. We want to construct an equivalence between $R$ and $\tmf(\Gamma)_{\Q}$. By the preceding corollary, we know that $\OO^{top}_{\Q}$ on $\MMb(\Gamma)$ is formal. In particular this provides us with compatible maps $R \to \OO^{top}(U)_{\Q}$ for all affines $U$ \'etale over $\MMb(\Gamma)$. Taking the homotopy limit, we obtain a map $R \to \Tmf(\Gamma)_{\Q}$. The uniqueness part of \cref{thm:tmfGamma} identifies $R$ with $\tmf(\Gamma)_{\Q}$. 
\end{proof}

\subsection{Not all $\tmf(\Gamma)$ are perfect}
While we have seen above that $\tmf(\Gamma)$ for a congruence group of level $n$ is always perfect as a $\tmf[\frac1n]$-module, we will see in this subsection that it is not necessarily compact as a $\tmf(\Gamma')[\frac1n]$-module for $\Gamma \subset \Gamma'$. The author has learned this argument from Tyler Lawson.
\begin{lemma}
 For $R = \tmf(\Gamma)_{\Q}$, the $R$-module $H\pi_0R$ can only be perfect if $\pi_*R$ is regular. 
\end{lemma}
\begin{proof}
 By \cite[Thm 19.1, Cor 19.5, Thm 19.12]{Eis95}, $\pi_*R$ is regular if and only if the graded $\Q$-vector space $\Tor_*^{\pi_*R}(\pi_0R,\pi_0R)$ is concentrated in finitely many dimensions. Because $R$ is formal by \cref{prop:GammaFormal}, this Tor agrees with $\pi_*(H\pi_0R\tensor_R H\pi_0R)$. Clearly, $H\pi_0R$ being a perfect $R$-module would imply the finite-dimensionality of this quantity. 
\end{proof}

It is actually very rare that $\pi_*\tmf(\Gamma)_{\Q} \cong M_*(\Gamma; \Q)$ is regular. One of the few exceptions is $\Gamma = \Gamma_1(3)$, where we obtain the ring $\Q[a_1, a_3]$. In contrast for $\Gamma = \Gamma_0(3)$, we obtain its $C_2$-fixed points, i.e.\ $\Q[a_1^2, a_3^2, a_1a_3] \cong \Q[x,y,z]/xz -y^2$, which is not regular. Thus, $H\Q$ is a perfect $\tmf_1(3)$-module, but is by the previous lemma not a perfect $\tmf_0(3)_{\Q}$-module. We obtain: 

\begin{prop}
 The $\tmf_0(3)$-module $\tmf_1(3)$ is not perfect, not even rationally. 
\end{prop}

\subsection{All $\tmf(\Gamma)$ are faithful}
The goal of this section is to show that if $\Gamma$ is a congruence subgroup of level $n$, then $\tmf(\Gamma)$ is (if defined) a faithful $\tmf[\frac1n]$-module, i.e.\ tensoring with it is conservative.

\begin{lemma}
For every congruence subgroup $\Gamma$ of level $n$, the $\Tmf[\frac1n]$-module $\Tmf(\Gamma)$ is faithful.\end{lemma}
\begin{proof}
By \cite{MM15}, the derived stack $(\MMb_{ell}, \OO^{top})$ is $0$-affine, i.e.\ the global sections functor 
\[\Gamma\colon \QCoh(\MMb_{ell}, \OO^{top}) \to \Mod_{\Tmf} \]
is a symmetric monoidal equivalence and the same is true after inverting $n$. Thus our claim becomes equivalent to showing that tensoring with $f_*\OO^{top}_{\MMb(\Gamma)}$ for $f\colon \MMb(\Gamma) \to \MMb_{ell, \Z[\frac1n]}$ is conservative on $\QCoh(\MMb_{ell}, \OO^{top})$. This can be checked \'etale locally, where $f_*\OO^{top}_{\MMb(\Gamma)}$ is free of positive rank as $f$ is finite and flat (see e.g.\ \cite[Prop 2.4]{MeiDecMod}) of positive rank everywhere (as $\MMb_{ell,\Z[\frac1n]}$ is irreducible and $\MMb(\Gamma)$ not empty).  
\end{proof}

In the following we fix a congruence subgroup $\Gamma$ and a multiplicatively closed subset $S$ of $\Z$ such that $\tmf(\Gamma)_S$ is defined (i.e.\ $\Gamma$ is tame or of index $2$ in a tame $\Gamma$). 


\begin{prop}
    The $\tmf_S$-module $\tmf(\Gamma)_S$ is faithful for every congruence subgroup $\Gamma$. 
\end{prop}
\begin{proof}
Let $M \in \Mod_{\tmf_S}$ with $M\tensor_{\tmf_S} \tmf(\Gamma)_S = 0$. It suffices to show that $M_{(p)} = 0$ for all $p$ not in $S$. Consider the case $p=2$ and localize everything implicitly at $2$. As $\tmf_1(3)$ is faithful over $\tmf$ (see \cite[Theorem 4.10]{Mathom}), it suffices to show that $M' = M\tensor_{\tmf}\tmf_1(3)$ vanishes. Our assumption implies $(M\tensor_{\tmf} \Tmf) \tensor_{\Tmf} \Tmf(\Gamma) =0$, hence by the faithfulness of $\Tmf(\Gamma)$ also $M\tensor_{\tmf} \Tmf = 0$. Thus, $M' \tensor_{\tmf_1(3)}\Tmf_1(3) = 0$. Moreover, $\tmf(\Gamma) \tensor_{\tmf} H\Z$ is a faithful $H\Z$-module as its $\pi_0$ is a faithful $\Z$-module. Thus $M' \tensor_{\tmf_1(3)} H\Z \simeq M \tensor_{\tmf} H\Z = 0$. 

Recall now that $\pi_*\tmf_1(3) \cong \Z[a_1, a_3]$. The map $\tmf_1(3)[a_i^{-1}] \to \Tmf_1(3)[a_i^{-1}]$ is an equivalence for $i= 1,3$ since the cofiber of $\tmf_1(3) \to \Tmf_1(3)$ is coconnective. Thus the considerations above imply that $M'[a_1^{-1}], M'[a_3^{-1}]$ and $M'/(a_1, a_3)$ all vanish, which implies the vanishing of $M'$. 

The argument for $p =3$ is similar with $\tmf_1(2)$ in place of $\tmf_1(3)$ and for $p>3$ we can use $\tmf$ itself as $\pi_*\tmf[\frac16] \cong \Z[\frac16][c_4, c_6]$ is a polynomial ring. 
\end{proof}

\section{Splittings}\label{sec:splittings}
Our goal in this setting is to show that $\tmf_1(n)$ often splits $p$-locally into small pieces. 

Fixing a natural number $n\geq 2$ and a prime $p$ not dividing $n$, we will work throughout this section implicitly $p$-locally. We demand that $M(\Gamma_1(n), \Z_{(p)}) \to M(\Gamma_1(n); \F_p)$ is surjective. In general, this is a subtle condition, but it is for example always fulfilled if $n\leq 28$ (see \cite[Remark 3.14]{MeiDecMod}). Equivalently, we can ask that $H^1(\MMb_1(n); \omega) \cong \pi_1\Tmf_1(n)$ does not have $p$-torsion. We note that this leaves plenty of cases where $\pi_1\Tmf_1(n) \neq 0$ and hence $\tmf_1(n)$ is not the naive connective cover of $\Tmf_1(n)$, of which the smallest is $n=23$. 

 By Theorem 1.3 of \cite{MeiTopLevel}, we have a splitting
\begin{equation}\label{eq:splitting}\Tmf_1(n) \simeq \bigoplus_i \Sigma^{2n_i}R\end{equation}
of $\Tmf$-modules, where $R$ is $\Tmf_1(3)$, $\Tmf_1(2)$ or $\Tmf$, depending on whether the prime $p$ is $2$, $3$ or bigger than $3$. In this splitting all $n_i$ are nonnegative. 

\begin{thm}\label{thm:splitting}
 Under the conditions as above, we have a splitting
 $$\tmf_1(n) \simeq \bigoplus_i \Sigma^{2n_i}r,$$
 where $r = \tau_{\geq 0}R$. 
\end{thm}
\begin{proof}
 Consider the composition
 $$f\colon \bigoplus_i \Sigma^{2n_i}r \to \bigoplus_i \tau_{\geq 0} \Sigma^{2n_i}R \to \tau_{\geq 0}\Tmf_1(n).$$
 Here, the second map is just the connective cover of \eqref{eq:splitting} (using that $\tau_{\geq 0}$ commutes with direct sums) and the first map is the direct sum of the maps $\Sigma^{2n_i}r \simeq \tau_{\geq 2n_i} \Sigma^{2n_i}R \to \tau_{\geq 0} \Sigma^{2n_i}R$. Since all negative homotopy of $R$ is in odd degrees, we see that $f$ is an isomorphism on even homotopy groups. Moreover, the source has only homotopy groups in even degrees. 
 
 Recall that we defined $\tmf_1(n)$ as a pullback 
 \[
  \xymatrix{
  \tmf_1(n)\ar[r]\ar[d] & H\Z  \ar[d] \\
  \tau_{\geq 0}\Tmf_1(n) \ar[r] & \tau_{[0,1]} \Tmf_1(n),
  }
 \]
where we still localize implicitly everywhere at $p$. This implies a fiber sequence
 $$\tmf_1(n) \to \tau_{\geq 0}\Tmf_1(n) \to \Sigma H\pi_1\Tmf_1(n).$$
 
 To factor $f$ over $\tmf_1(n)$, it is enough to show that $H^1(\Sigma^{2n_i}r; A) = 0$ with any coefficients $A$. This is clear anyhow for $n_i\geq 1$, so assume $n_i=0$. We know that $\tau_{[0,1]}r \simeq H\Z$ and we have $H^1(H\Z; A) \cong H^1(\mathbb{S};A) = 0$ (as the the cofiber of $\S \to H\Z$ is $1$-connected). 
 
 Now $\pi_*\tmf_1(n)$ is concentrated in even degrees and $\tmf_1(n) \to \tau_{\geq 0}\Tmf_1(n)$ induces a $\pi_*$-isomorphism in even degrees. In total, we see that $f$ induces an isomorphism on $\pi_*$.  
\end{proof}

\begin{remark}
    The condition that $\pi_1\Tmf_1(n) \cong H^1(\MMb_1(n); \omega)$ does not have $p$-torsion is actually necessary in the preceding theorem. One can indeed show that $\Tmf_1(n)$ can be recovered as $\tmf_1(n) \tensor_{\tmf}\Tmf$. Thus a $p$-local $\tmf$-linear splitting of $\tmf_1(n)$ into shifted copies of $r$ implies a $p$-local splitting of $\Tmf_1(n)$ into copies of $R$. As the latter has torsionfree homotopy groups, such a splitting can indeed only occur if the homotopy groups of $\Tmf_1(n)$ are $p$-torsionfree as well.  
\end{remark}

We now fix $p=2$ and are thus assuming that $\pi_1\Tmf_1(n) \cong H^1(\MMb_1(n); \omega)$ does not have $2$-torsion (this is e.g.\ true for all odd $2\leq n<65$ by \cite[Remark 3.14]{MeiDecMod}). In this setting we also want to prove connective versions of the $C_2$-equivariant refinement
\begin{equation}\label{eq:C2splitting}\Tmf_1(n)_{(2)} \simeq_{C_2} \bigoplus_i \Sigma^{n_i\rho}\Tmf_1(3)_{(2)}\end{equation}
 of \eqref{eq:splitting} given in \cite[Theorem 6.19]{MeiTopLevel}, where $\rho$ is the regular representation of $C_2$. We need the following lemma:

\begin{lemma}\label{lem:EM}
Let $A$ be an abelian group without $2$-torsion and denote by $\underline{A}$ the corresponding constant $C_2$-Mackey functor. Then $\pi_{-\sigma}^{C_2}H\underline{A} \cong A\tensor \Z/2$ and the map 
$$[H\underline{\Z}, \Sigma^\sigma H\underline{A}]^{C_2}\xrightarrow{\pi_0^{C_2}} A \tensor \Z/2$$
is an isomorphism. 
\end{lemma}
\begin{proof}
	Smashing the fundamental cofiber sequence 
	\[(C_2)_+ \to S^0 \to S^{\sigma} \to \Sigma(C_2)_+\]
	 with $S^{-\sigma}$ and mapping out of it yields an exact sequence
	\[\pi_{-1}^eH\underline{A} \leftarrow \pi_{-\sigma}^{C_2}H\underline{A} \leftarrow \pi_{0}^{C_2}H\underline{A} \leftarrow  \pi_0^eH\underline{A}.\]
	The rightmost arrow can be identified with the transfer $\mathrm{tr} = 2\colon A \to A$ of the constant  Mackey functor, while $\pi_{-1}^eH\underline{A} = 0$. We obtain $\pi_{-\sigma}^{C_2}H\underline{A} \cong A\tensor \Z/2$ as claimed.  
	
	To finish the proof, we recall from \cref{sec:C2argument} that $\tau_{\leq 1}C\overline{\eta} \simeq H\underline{\Z}$. As $\Sigma^{\sigma}H\underline{A}\leq 1$ in the slice filtration, this implies that $[H\underline{\Z}, \Sigma^{\sigma}H\underline{A}]^{C_2} \cong [C\overline{\eta}, \Sigma^{\sigma} H\underline{A}]^{C_2}$. This sits in a long exact sequence
	\[0 = \pi_1^{C_2}H\underline{A} \to [C\overline{\eta}, \Sigma^{\sigma}H\underline{A}] \to \pi_{-\sigma}^{C_2}H\underline{A}\to \pi_0^{C_2}H\underline{A} = A.\]
	As $A$ does not have $2$-torsion and we have shown above that $\pi_{-\sigma}^{C_2}H\underline{A} \cong A\tensor \Z/2$, the result follows. 
\end{proof}

\begin{thm}\label{thm:splittingC2}
 Assuming that $n\geq 3$ is odd and $H^1(\MMb_1(n); \omega)$ does not have $2$-torsion, we have $2$-locally a $C_2$-equivariant splitting
 $$\tmf_1(n) \simeq \bigoplus_i \Sigma^{n_i\rho}\tmf_1(3).$$
\end{thm}
\begin{proof}
We localize everywhere implicitly at $2$ and consider the map 
$$\bigoplus_i  \Sigma^{n_i\rho}\tmf_1(3) \to \bigoplus_i \tau_{\geq 0}\Sigma^{n_i\rho} \Tmf_1(3) \xrightarrow{\tau_{\geq 0}\Phi} \tau_{\geq 0}\Tmf_1(n),$$
for a chosen $C_2$-equivalence $\Phi$ between $\bigoplus_i \Sigma^{n_i\rho} \Tmf_1(3)$ and $\Tmf_1(n)$. 
We have a fiber sequence
$$\tmf_1(n) \to \tau_{\geq 0}\Tmf_1(n) \to \Sigma^\sigma H\underline{A},$$
where $A= H^1(\MMb_1(n);\omega)$ since $\Sigma^\sigma H\underline{A}$ is the $1$-slice of $\Tmf_1(n)$ by \cite[Theorem 6.16]{MeiTopLevel}. On $\pi_0^{C_2}$ this induces (using \cref{lem:EM}) a short exact sequence
\begin{equation}\label{eq:pi0sequence}0 \to \Z \to \pi_0^{C_2}\Tmf_1(n) \xrightarrow{r} A \tensor \Z/2 \to 0.\end{equation}
The composite $\bigoplus  \Sigma^{n_i\rho}\tmf_1(3)  \to \Sigma^\sigma H\underline{A}$ factors over the $1$-slice coconnective cover of the source, which agrees with $H\underline{\Z}$ since there is precisely one $n_i$ equalling $0$ (by considering non-equivariant homotopy groups). Using \cref{lem:EM} again, the resulting map $H\underline{\Z} \to \Sigma^\sigma H\underline{A}$ is null iff the image $r(\Phi(1))$ of $\Phi(1)$ in $A\tensor \Z/2$ is $0$. 

We want to show that we can change $\Phi$ so that this is true. Using $\Phi$, the $C_2$-spectrum $\Tmf_1(n)$ gets the structure of a $\Tmf_1(3)$-module. Thus, $\Tmf_1(3)$-module maps $$\bigoplus_{i=0}^{N} \Sigma^{n_i\rho}\Tmf_1(3) \to \Tmf_1(n)$$ correspond to a sequence of classes $x_i \in \pi_{n_i\rho}^{C_2}\Tmf_1(n)$ by considering the images of $1\in \pi_{n_i\rho}^{C_2}\Sigma^{n_i\rho}\Tmf_1(3)$. Denote the sequence corresponding to $\Phi$ by $e_0, \dots, e_N$. By possibly reordering, we can assume $n_0= 0$. We construct a new map $\Phi'\colon \bigoplus_{i=0}^{N} \Sigma^{n_i\rho}\Tmf_1(3) \to \Tmf_1(n)$ corresponding to $x_0, x_1, \dots, x_N$ with $x_i = e_i$ for $i>0$ and $x_0$ corresponding to the image of $u\in\Z$ in \eqref{eq:pi0sequence}, where $u$ maps to $\res_{e}^{C_2}(e_0)$ along the isomorphism $\Z \cong \pi_0^e\tmf_1(n) \to \pi_0^e\Tmf_1(n)$. As $\Phi'$ and $\Phi$ induce the same map on underlying homotopy groups, the map $\Phi'$ is an equivalence. By construction, $r(x_0) = 0$. 

Thus the map 
\[\bigoplus_i\Sigma^{n_i\rho}\tmf_1(3) \to \bigoplus_i \tau_{\geq 0}\Sigma^{n_i\rho}\tmf_1(3)\xrightarrow{\tau_{\geq 0}\Phi'} \tau_{\geq 0} \Tmf_1(n)\] 
factors indeed over $\tmf_1(n)$. As before, the map $\Sigma^{n_i\rho}\tmf_1(3) \to \tmf_1(n)$ induces an isomorphism on underlying homotopy groups. Both source and target are strongly even and thus the map is a $C_2$-equivariant equivalence by \cite[Lemma 3.4]{H-M17}.
\end{proof}

\bibliographystyle{alpha}
\bibliography{chromatic}
\end{document}